\documentclass{amsart} 
\usepackage{amsthm, amsmath, mathtools}
\usepackage{amssymb}
\usepackage{amscd}
\usepackage{enumerate}
\usepackage{mathrsfs} 
\usepackage[curve,matrix,arrow,frame,tips]{xy}
\usepackage{colonequals}
\usepackage{verbatim}
\usepackage{pdfsync}
\usepackage{graphicx}
\usepackage[usenames,dvipsnames]{color}

\usepackage{url}

\usepackage{fullpage}

\usepackage[normalem]{ulem}

\numberwithin{equation}{section} 

\theoremstyle{plain}
\newtheorem{theorem}{Theorem}[section]
\newtheorem{proposition}[theorem]{Proposition}
\newtheorem{lemma}[theorem]{Lemma}

\newtheorem{corollary}[theorem]{Corollary}

\theoremstyle{definition}
\newtheorem{definition}[theorem]{Definition}

\newtheorem{example}[theorem]{Example}
\newtheorem{hypo}[theorem]{Hypothesis}

\theoremstyle{remark}
\newtheorem{remark}[theorem]{Remark}

\newcommand{\E}{{\mathcal E}}

\newcommand{\F}{\mathbb{F}}

\newcommand{\Q}{\mathbb{Q}}

\newcommand{\Z}{\mathbb{Z}}

\DeclareMathOperator{\Spec}{Spec}

\DeclareMathOperator{\SL}{SL}
\DeclareMathOperator{\GL}{GL}

\DeclareMathOperator{\Gal}{Gal}

\renewcommand{\to}{\longrightarrow}

\newcommand{\HH}{\mathrm{H}}

\usepackage{microtype}

\begin{document}

\title{Massey products and elliptic curves}

\author[F. Bleher]{Frauke M. Bleher}
\address{F.B.: Department of Mathematics\\University of Iowa\\
14 MacLean Hall\\Iowa City, IA 52242-1419\\ U.S.A.}
\email{frauke-bleher@uiowa.edu}
\thanks{The first author was supported in part by NSF  Grant No.\ DMS-1801328 and Simons Foundation grant No. 960170.}

\author[T. Chinburg]{Ted Chinburg}
\address{T.C.: Department of Mathematics\\University of Pennsylvania\\
Philadelphia, PA 19104-6395\\ U.S.A.}
\email{ted@math.upenn.edu}
\thanks{The second author was supported in part by NSF SaTC grant No. CNS-1701785.}

\author[J. Gillibert]{Jean Gillibert}
\address{J. G.:  Institut de Math{\'e}matiques de Toulouse \\ CNRS UMR 5219 \\
118, route de Narbonne, 31062 Toulouse Cedex \\ France.}
\email{Jean.Gillibert@math.univ-toulouse.fr}
\thanks{The third author is the corresponding author. He was supported in part by the CIMI Labex.}

\date{April 27, 2023}

\subjclass[2010]{14F20 (Primary) 55S30, 14H52 (Secondary)}
\keywords{Massey products, \'etale cohomology, smooth projective curves, elliptic curves}

\begin{abstract}
We study the vanishing of Massey products of order at least $3$ for absolutely irreducible smooth projective curves  over a field  with coefficients in $\mathbb{Z}/\ell$. We mainly focus on elliptic curves, for which we obtain a complete characterization of when triple Massey products do not vanish.
\end{abstract}

\maketitle


\section{Introduction}
\label{s:intro}

This paper has to do with the vanishing of triple Massey products on $\HH^1(X,\mathbb{Z}/\ell)$ when $\ell$ is an odd prime and $X$ is an absolutely irreducible smooth projective variety over a  field $F$ in which $\ell$ is invertible.  When $d = \mathrm{dim}(X) = 0$, Min{\'a}\v{c} and T{\^a}n showed in \cite{MinacTan2016} that this triple product always vanishes for arbitrary $F$, following earlier work by
 Hopkins and Wickelgren \cite{HopkinsWickelgren}, Matzri \cite{Matzri}, Efrat and Matzri \cite{EfratMatzri}, and others.  When $d = 0$ and $F$ is a number field, Harpaz and Wittenberg showed in \cite{HarWit2019} that all Massey products of order at least $3$ vanish.  For a more detailed account of the case $d  = 0$ see the introduction of \cite{HarWit2019}. 
Ekedahl gave an example in \cite{Ekedahl} showing that the triple Massey product need not vanish when $d = 2$ and $F = \mathbb{C}$.   

The present paper arose from the  problem of determining when triple Massey  products vanish when $d = 1$, i.e. for curves over an arbitrary field $F$.  Our main result classifies exactly which triple Massey products do not vanish when $X = E$ is an elliptic curve over $F$ and the $\ell$-torsion of $E$ over a separable closure $\bar{F}$ of $F$ is defined over $F$.  We show that the only case in which these do not vanish is when $\ell = 3$
 and the three elements of $\HH^1(X,\mathbb{Z}/\ell)$ generate the same one-dimensional space (see Lemma \ref{lem:nice!}).   The classification when $\ell = 3$ of non-vanishing triple Massey products is given in Theorem \ref{thm:classifiynontrivial!}.  One consequence is the following result, where  $\bar{E}=E\otimes_F \bar{F}$ and $G_F^{(3)}$ denotes the pro-$3$ completion of $\mathrm{Gal}(\bar{F}/F)$:
 
 \begin{theorem}
\label{cor:nontrivialgeneral!}
Let $F$ be a field whose characteristic is not $3$, and let $E$ be an elliptic curve over $F$ such that the $3$-torsion of $E(\bar{F})$ is defined over $F$. 
There exists a character $\chi\in\HH^1(E,\mathbb{Z}/3)=\mathrm{Hom}(\pi_1(E),\mathbb{Z}/3)$ such that $\langle \chi,\chi,\chi \rangle$ does not contain zero if and only if either 
\begin{itemize}
\item[(i)] the action of $G_F^{(3)}$ on $\bar{E}[9]$ is not given by multiplication by scalars in $(\mathbb{Z}/9)^\times$, or 
\item[(ii)] the action of $G_F^{(3)}$ on $\bar{E}[9]$ is given by multiplication by scalars in $(\mathbb{Z}/9)^\times$ and there exists a primitive ninth root of unity $\zeta\in\bar{F}$ such that $\zeta\not\in F$ and $F(\zeta)$ is not the only cubic extension of $F$ inside $\bar{F}$.
\end{itemize}
\end{theorem}

As an explicit example, suppose $F$ is a number field containing $\mathbb{Q}(\sqrt{3},\sqrt{-1})$ that does not contain a primitive ninth root of unity. When $E$ is the elliptic curve over $F$ with model $y^2 = x^3 - 1$, there is a character $\chi \in \HH^1(E,\mathbb{Z}/3)$ with non-vanishing triple Massey product (see Example \ref{ex:littleexample}).

Regarding primes $\ell>3$, we prove the following:

\begin{theorem}
\label{cor:forallell}
Let $\ell > 3$ be a prime number. Then there exist a prime number $p\neq \ell$, an elliptic curve $E$ defined over $\F_p$, and
 non-trivial characters $\chi_1,\chi_2,\chi_3\in\HH^1(E,\mathbb{Z}/\ell)$ such that $\langle \chi_1,\chi_2,\chi_3\rangle$ is not empty and does not contain zero.
\end{theorem}

We now describe the contents of the paper.  In \S \ref{s:prelim} we recall some basic results about Galois and \'etale cohomology and about Massey products.  In \S \ref{s:restrict} we show that higher Massey products on curves over separably closed fields always contain $0$ provided they are not empty. In \S \ref{s:necessary} we prove some necessary conditions for higher Massey products to be non-empty and to not contain $0$.  In \S \ref{s:triple} we prove our main results, Theorem \ref{thm:classifiynontrivial!} and Theorem  \ref{cor:nontrivialgeneral!}, concerning triple Massey products when $X=E$ is an elliptic curve.  In \S \ref{s:ellipticnumber} we analyze two families of examples arising from specializing a generic family of elliptic curves \eqref{eq:igusacurve} and from CM elliptic curves. Finally in \S \ref{s:ellipticfinite} we treat arbitrary elliptic curves  $E$ when $F$ is a finite field without the assumption that the $\ell$-torsion of $E(\bar{F})$ is defined over $F$. We conclude by proving Theorem~\ref{cor:forallell}.

\medbreak

\noindent
\textbf{Acknowledgements.} The first and second authors would like to thank the University of Toulouse for its support and hospitality during work on this paper.   
We would like to thank Florian Pop and the anonymous referee for their comments and suggestions which helped improve the paper.


\section{Preliminaries}
\label{s:prelim}

Let $F$ be a field with a fixed separable closure $\bar{F}$ inside a fixed algebraic closure $F^{\mathrm{alg}}$, and let $\ell$ be a prime number that is invertible in $F$. Let $X$ be a smooth projective geometrically irreducible curve over $F$, and define 
$$\quad \bar{X}:=X\otimes_F\bar{F}.$$ 
Let $\eta$ be a geometric point of $X$, i.e. a point with values in $\bar{F}$, which we also view as a geometric point of $\bar{X}$. To simplify notation, we denote the \'etale fundamental groups by
$$\pi_1(X):=\pi_1(X,\eta) \quad \text{and}\quad  \pi_1(\bar{X}):=\pi_1(\bar{X},\eta).$$

\begin{remark}
\label{rem:voodoo}
One may ask what happens if one replaces the separable closure $\bar{F}$ by the algebraic closure $F^\mathrm{alg}$. Since $F^{\mathrm{alg}}/\bar{F}$ is purely inseparable, the restriction functor gives an equivalence of categories between the small \'etale sites of $\bar{X}$ and of $X\otimes_F F^{\mathrm{alg}}$ \cite[VIII, Th{\'e}or{\`e}me~1.1 and Exemples~1.3]{SGA4-2}. Since $\ell$ is invertible in $F$, we obtain for all $i\ge 0$,
$$\HH^i(\bar{X},\mathbb{Z}/\ell)=\HH^i(X\otimes_F F^{\mathrm{alg}},\mathbb{Z}/\ell).$$

Similarly, according to \cite[VIII, \S{}3.4]{SGA4-2}, the fiber functors relative to the geometric points $\eta\otimes_F \Spec(F^\mathrm{alg})\to X\otimes_F F^{\mathrm{alg}}$ and $\eta\to \bar{X}$ are isomorphic. Therefore $\pi_1(\bar{X},\eta) = \pi_1(X\otimes_F F^{\mathrm{alg}}, \eta\otimes_F F^\mathrm{alg})$.
In particular, the maximal elementary abelian $\ell$-quotient groups of these groupes are the same, which implies
$$\mathrm{Pic}(\bar{X})[\ell] = \mathrm{Pic}(X\otimes_F F^{\mathrm{alg}})[\ell].$$

Moreover, by \cite[Prop. I.3.24]{Milne}, $\bar{X}$ and $X\otimes_F F^{\mathrm{alg}}$ are smooth projective curves over $\bar{F}$ and $F^{\mathrm{alg}}$, respectively, and $F^{\mathrm{alg}}\otimes_FF(X)$ is a field. By \cite[Thm. 26.2]{Matsumura}, $F(X)$ is separably generated over $F$, i.e. there is an element $t\in F(X)$ that is transcendental over $F$ such that $F(X)/F(t)$ is a separable algebraic extension. Therefore, $\bar{X}(\bar{F})$ is Zariski dense in $\bar{X}$.
\end{remark}

We  have a natural isomorphism of first cohomology groups
$$\HH^1(X,\mathbb{Z}/\ell) \cong \HH^1(\pi_1(X),\mathbb{Z}/\ell).$$
For higher cohomology groups, we have the following result from \cite[\S2.1.2]{AchingerThesis} (see also \cite[\S3]{Achinger2015}):

\begin{proposition}
\label{prop:Achinger}
If $\bar{X}$
is not isomorphic to $\mathbb{P}^1_{\bar{F}}$ then there is a natural isomorphism
$$\HH^i(X,\mathbb{Z}/\ell) \cong \HH^i(\pi_1(X),\mathbb{Z}/\ell)$$
for all $i\ge 1$.
\end{proposition}

For the remainder of the paper we assume that $\bar{X}$ is not isomorphic to $\mathbb{P}^1_{\bar{F}}$, and we identify $\HH^i(X,\mathbb{Z}/\ell) = \HH^i(\pi_1(X),\mathbb{Z}/\ell)$ for all $i\ge 1$.

We use the following definition of Massey products from \cite[\S1]{Kraines1996},
which differs from Dwyer's definition in \cite{Dwyer1975} by a sign.

\begin{definition}
\label{def:Massey}
Let $t\ge 2$ be an integer, and let $\chi_1,\ldots,\chi_t\in \HH^1(X,\mathbb{Z}/\ell)=\mathrm{Hom}(\pi_1(X),\mathbb{Z}/\ell)$. 
The $t$-fold Massey product $\langle \chi_1,\ldots,\chi_t\rangle$ is the subset of $\HH^2(X,\mathbb{Z}/\ell)$ consisting of the classes of all $2$-cocycles $\nu$ for which there exists a collection of continuous maps $\kappa_{i,j}: \pi_1(X) \to\mathbb{Z}/\ell$, $1\le i\le j\le t$, $(i,j)\ne (1,t)$, such that
\begin{itemize}
\item[(i)] $\kappa_{i,i}=\chi_i$ for $1\le i\le t$, and
\item[(ii)] $(\delta \kappa_{i,j})(\sigma,\tau)=-\sum_{r=i}^{j-1} \kappa_{1,r}(\sigma)\,\kappa_{r+1,j}(\tau)$ for all $\sigma,\tau\in\pi_1(X)$, when $1\le i< j\le t$, $(i,j)\ne (1,t)$, and
\item[(iii)] $\nu(\sigma,\tau)=-\sum_{r=1}^{t-1} \kappa_{1,r}(\sigma)\,\kappa_{r+1,t}(\tau)$ for all $\sigma,\tau\in\pi_1(X)$.
\end{itemize}
Any collection of continuous maps $\{\kappa_{i,j}\}$ satisfying (i)-(iii) is called a defining system for $\langle \chi_1,\ldots,\chi_t\rangle$.  
\end{definition}

Massey products generalize cup products, since if $t=2$ then the only defining system for $\langle \chi_1,\chi_2\rangle$ is $\{\chi_1,\chi_2\}$, and $\langle \chi_1,\chi_2\rangle = \{- \chi_1\cup\chi_2\}$.

The definition of Massey products can be motivated as follows. Let $U_{t+1}(\mathbb{Z}/\ell)$ be the group of upper triangular unipotent (a.k.a. unitriangular) matrices of size $(t+1)\times(t+1)$ with entries in $\mathbb{Z}/\ell$,  and let $Z(U_{t+1}(\mathbb{Z}/\ell))$ be its center, which consists of the unitriangular matrices for which all entries above the main diagonal that are not at the $(1,t+1)$ position are zero.  The data of a defining system is equivalent to giving a continuous group homomorphism 
\begin{equation}
\label{eq:rhobar}
\begin{array}{cccc}
\vartheta:\quad &\pi_1(X) &\to& U_{t+1}(\mathbb{Z}/\ell)/Z(U_{t+1}(\mathbb{Z}/\ell))\\
&\sigma & \mapsto & \left(\begin{array}{ccccc}
1&\kappa_{1,1}(\sigma)&\cdots & \kappa_{1,t-1}(\sigma)&*\\
0&1&\kappa_{2,2}(\sigma)&\cdots&\kappa_{2,t}(\sigma)\\
\vdots&\ddots&\ddots&\ddots&\vdots\\
0&\cdots&0&1&\kappa_{t,t}(\sigma)\\
0&\cdots&\cdots&0&1\end{array}\right) 
\end{array}
\end{equation}
with $\kappa_{i,i}$ being the character $\chi_i$ for all $i$.  There is a continuous group homomorphism $\rho=\rho(\vartheta): \pi_1(X) \to U_{t+1}(\mathbb{Z}/\ell)$ lifting $\vartheta$ if and only if the $2$-cocycle $\nu$ is the coboundary of a continuous function $\kappa_{1,t}:\pi_1(X) \to \mathbb{Z}/\ell$.  Thus  $\langle \chi_1,\ldots,\chi_t\rangle$ is not empty if and only if a homomorphism $\vartheta$ as in (\ref{eq:rhobar}) exists,
and $\langle \chi_1,\ldots,\chi_t\rangle$ contains $0$ if and only if there is such a $\vartheta$ that has a lift $\rho(\vartheta)$ to $U_{t+1}(\mathbb{Z}/\ell)$.  For more details, see \cite[Thm. 2.4]{Dwyer1975} and also \cite[Lemma 4.2]{MinacTan2015}.

We now summarize some useful properties of Massey products.

Let $t\ge 2$, and let $\chi_1,\ldots,\chi_t\in \HH^1(X,\mathbb{Z}/\ell)$ be non-zero characters such that the $t$-fold Massey product $\langle \chi_1,\ldots,\chi_t\rangle$ is not empty. 

It follows from Definition \ref{def:Massey} and its connection to continuous group homomorphisms as in (\ref{eq:rhobar}) that $\langle \chi_1,\chi_2,\ldots,\chi_e\rangle$ contains $0$ for all $2\le e\le t-1$, and $\langle \chi_s,\ldots,\chi_{t-1},\chi_t\rangle$ contains $0$ for all $2\le s\le t-1$.

By \cite[(2.3)]{Kraines1996}, we have for all $c_1,\ldots,c_t\in \mathbb{Z}/\ell$ that
\begin{equation}
\label{eq:Masseyscalar}
\langle c_1\chi_1,\ldots,c_t\chi_t\rangle \supseteq (c_1\cdots c_t)\,\langle \chi_1,\ldots,\chi_t\rangle.
\end{equation}
Now suppose that all $c_1,\ldots,c_t$ are non-zero. Writing $\chi_i=c_i^{-1}(c_i\chi_i)$, for $1\le i\le t$, and applying (\ref{eq:Masseyscalar}) again, we obtain
\begin{equation}
\label{eq:Masseyscalar!}
\langle c_1\chi_1,\ldots,c_t\chi_t\rangle = (c_1\cdots c_t)\,\langle \chi_1,\ldots,\chi_t\rangle.
\end{equation}
In particular, if $\chi\in\HH^1(X,\mathbb{Z}/\ell)$ is a single non-zero character and $a_1,\ldots,a_t\in (\mathbb{Z}/\ell)^\times$, then
\begin{equation}
\label{eq:ohyeah!}
\langle a_1\chi,\ldots,a_t\chi\rangle \quad\mbox{contains $0$} \quad\mbox{if and only if}\quad 
\langle \,\underbrace{\chi,\ldots,\chi}_{t\;\mathrm{copies}}\,\rangle\quad\mbox{contains $0$.}
\end{equation}

Because of (\ref{eq:ohyeah!}), it is useful to introduce a ``restricted" $t$-fold Massey product when all characters are the same (see \cite[\S3]{Kraines1996}). Namely, when all of $\chi_1,\ldots,\chi_t$ equal a single character $\chi$, then the restricted $t$-fold Massey product 
\begin{equation}
\label{eq:Masseyrestrict}
\langle \chi \rangle^t \subseteq \langle \,\underbrace{\chi,\ldots,\chi}_{t\;\mathrm{copies}}\, \rangle
\end{equation} 
is defined to be the subset of $\HH^2(X,\mathbb{Z}/\ell)$ consisting of the classes of all $2$-cocycles $\nu$ as in Definition \ref{def:Massey} that are associated to defining systems for which the functions $\kappa_{i,j}$ only depend on $i+j$, for $1\le i<j\le t$, $(i,j)\ne (1,t)$.
If $t=\ell$, then it follows from \cite[Thm. 14]{Kraines1996} that $\langle \chi \rangle^\ell$ is non-empty and a singleton given by
\begin{equation}
\label{eq:Bockstein}
\langle \chi \rangle^\ell = \{-\beta(\chi)\}
\end{equation}
where $\beta$ is the Bockstein operator associated to the exact sequence
$$0\to\mathbb{Z}/\ell \to \mathbb{Z}/\ell^2\to \mathbb{Z}/\ell\to 0.$$

In later sections, we will focus on triple Massey products. These are easier to describe than general Massey products, as highlighted in the following remark.

\begin{remark}
\label{rem:tripleMassey}
Let $\chi_1,\chi_2,\chi_3\in \HH^1(X,\mathbb{Z}/\ell)$. Then the triple Massey product $\langle \chi_1,\chi_2,\chi_3\rangle$ is not empty if and only if $\chi_1 \cup \chi_2 = 0 = \chi_2 \cup \chi_3$ in $\HH^2(X,\mathbb{Z}/\ell)$. Suppose $\kappa=\{\kappa_{1,1},\kappa_{1,2},\kappa_{2,2},\kappa_{2,3},\kappa_{3,3}\}$ is a defining system for $\langle \chi_1,\chi_2,\chi_3\rangle$. In particular, $\kappa_{i,i}=\chi_i$ for $1\le i\le 3$. Then all defining systems can be obtained from $\kappa$ by adding a continuous homomorphism $f_{1,2}:\pi_1(X)\to\mathbb{Z}/\ell$ to $\kappa_{1,2}$ or by adding a continuous homomorphism $f_{2,3}:\pi_1(X)\to\mathbb{Z}/\ell$ to $\kappa_{2,3}$ (or both). This means that the 2-cocycle $\nu$, with
$$\nu(\sigma,\tau) = - \kappa_{1,1}(\sigma)\kappa_{2,3}(\tau) - \kappa_{1,2}(\sigma)\kappa_{3,3}(\tau)=-\chi_1(\sigma)\kappa_{2,3}(\tau) - \kappa_{1,2}(\sigma)\chi_3(\tau)$$
for all $\sigma,\tau\in\pi_1(X)$, gives a single well-defined element $[\nu]$ in the quotient group 
\begin{equation}
\label{eq:quotienttriple}
\frac{\HH^2(X,\mathbb{Z}/\ell)}{\HH^1(X,\mathbb{Z}/\ell)\cup \chi_3 + \chi_1 \cup \HH^1(X,\mathbb{Z}/\ell)}.
\end{equation}
In particular, $\langle \chi_1,\chi_2,\chi_3\rangle$ contains $0$ if and only if $[\nu]$ is the identity element of the quotient group (\ref{eq:quotienttriple}).
\end{remark}

In the next remark we summarize some important properties of the group $U_4(\mathbb{Z}/\ell)$ of unitriangular $4\times 4$ matrices over $\mathbb{Z}/\ell$ that we will need in later sections.

\begin{remark}
\label{rem:U4!}
Let $\ell\ge 3$ and let
\begin{equation}
\label{eq:matrixnotation}
M=M(a_1,a_2,a_3,u,v,w):= \begin{pmatrix} 
1&a_1&u&v\\
0&1&a_2&w\\
0&0&1&a_3\\ 
0&0&0&1
\end{pmatrix}
\end{equation}
in $U_4(\mathbb{Z}/\ell)$ for $a_1,a_2,a_3,u,v,w\in\mathbb{Z}/\ell$. It is well known that 
\begin{equation}
\label{eq:powers}
M^\ell = \begin{pmatrix} 
1&\ell a_1&\ell u + \binom{\ell}{2} a_1a_2 &
\ell v + \binom{\ell}{2} a_1w + \binom{\ell}{2} a_3u + \binom{\ell}{3} a_1a_2a_3\\
0&1&\ell a_2&\ell w + \binom{\ell}{2} a_2a_3\\
0&0&1&\ell a_3\\ 
0&0&0&1\end{pmatrix}.
\end{equation}
In particular, if $\ell>3$ then every non-identity element of $U_4(\mathbb{Z}/\ell)$ has order $\ell$. On the other hand, $U_4(\mathbb{Z}/3)$ contains elements of order 9, which are precisely the matrices $M$ in (\ref{eq:matrixnotation}) with $a_1a_2a_3\ne 0$. 

For $\ell\ge 3$, we have the following formula of the commutator
$[M,\tilde{M}]=M\tilde{M}M^{-1}\tilde{M}^{-1}$ when 
$M$ is above and $\tilde{M}=M(\tilde{a}_1,\tilde{a}_2,\tilde{a}_3,\tilde{u},\tilde{v},\tilde{w})$:
\begin{equation}
\label{eq:commutator}
[M,\tilde{M}] = 
\begin{pmatrix}
1&0&a_1\tilde{a}_2-a_2\tilde{a}_1&
(a_1\tilde{w}-w\tilde{a}_1)-(a_3\tilde{u}-\tilde{a}_3u) - (a_1\tilde{a}_2-a_2\tilde{a}_1)(a_3 + \tilde{a}_3)
\\
0&1&0&a_2\tilde{a}_3-a_3\tilde{a}_2\\
0&0&1&0\\ 
0&0&0&1
\end{pmatrix}.
\end{equation}
In particular, the commutator subgroup $H_1$ of $U_4(\mathbb{Z}/\ell)$ is the subgroup of matrices $M$ as  in (\ref{eq:matrixnotation}) for which $a_1=a_2=a_3 = 0$. It follows from (\ref{eq:commutator}) that $H_1$ is an abelian subgroup of $U_4(\mathbb{Z}/\ell)$, which means that the second derived subgroup of $U_4(\mathbb{Z}/\ell)$ is trivial. As usual, the center of $U_4(\mathbb{Z}/\ell)$ consists of all matrices $M$ as in (\ref{eq:matrixnotation}) with $a_1=a_2=a_3=u=w=0$.

When $\ell=3$, we will also need the subgroup $H$ of $U_4(\mathbb{Z}/3)$ consisting of all matrices of the form
\begin{equation}
\label{eq:matrixnotation2}
N=N(a,u,v,w):= \begin{pmatrix} 
1&a&u&v\\
0&1&a&w\\
0&0&1&a\\ 
0&0&0&1
\end{pmatrix}.
\end{equation}
In particular, $H_1\le H$. It follows that the matrices of order $9$ in $H$ are precisely the matrices in $H-H_1$. By (\ref{eq:commutator}), the center of $H$ consists of all matrices $N$ as in (\ref{eq:matrixnotation2}) with $a=0$ and $u=w$. Moreover, the commutator subgroup $H'$ of $H$ is the subgroup of $Z(H)$ consisting of all matrices $N$ as in (\ref{eq:matrixnotation2}) with $a=0$ and $u=w=0$.
\end{remark}

The following result on cup products will be important in the next sections.

\begin{lemma}
\label{lem:cupp} 
Let $\chi,\psi\in \HH^1(X,\mathbb{Z}/\ell)$ with restrictions $\bar{\chi},\bar{\psi}$ to $\HH^1(\bar{X},\mathbb{Z}/\ell)$. Suppose $\bar{\chi}=0$, $\chi\ne 0$, and $\chi\cup\psi=0$. If the $\ell$-torsion $\mathrm{Pic}(\bar{X})[\ell]$ is defined over $F$, then $\bar{\psi}=0$. 
\end{lemma}

\begin{proof}
 Suppose by way of contradiction that $\bar{\psi}\ne 0$.   Since $\chi\cup\psi=0$, there exists a continuous function $\kappa: \pi_1(X)\to \mathbb{Z}/\ell$ such that 
$$\begin{array}{cccc}
\rho:& \pi_1(X) &\to& U_3(\mathbb{Z}/\ell)\\
&g & \mapsto & {\left(\begin{array}{ccc}
1&\chi(g)&\kappa(g)\\
0&1&\psi(g)\\
0&0&1\end{array}\right) }
\end{array}$$
is a continuous group homomorphism. Since $\bar{\chi}=0$ but $\bar{\psi}\ne 0$, $\psi$ is not in the group of characters generated by $\chi$.  Hence the image  of $\rho$ surjects onto the quotient of $U_3(\mathbb{Z}/\ell)$ by its center.  If $\rho$ is not surjective then the image of $\rho$  has order $\ell^2$ so is abelian.  Since $U_3(\mathbb{Z}/\ell)$ is generated by this image and its center, this would force $U_3(\mathbb{Z}/\ell)$ to  be abelian, which is  a contradiction.  So $\rho$ is surjective, and in particular the image of $\rho$ is not abelian.

Since $\bar{\chi}=0$, the image $\rho(\pi_1(\bar{X}))$ lies in an elementary abelian $\ell$-subgroup of $U_3(\mathbb{Z}/\ell)$. Since $\pi_1(\bar{X})$ is a normal subgroup of $\pi_1(X)$ and $\rho$ is surjective, it follows that $\rho(\pi_1(\bar{X}))$ is a normal subgroup of $U_3(\mathbb{Z}/\ell)$. Moreover, $\bar{\chi}=0$ implies that $\#\rho(\pi_1(\bar{X}))\le \ell^2$. On the other hand, $\bar{\psi}\ne 0$ implies that $\#\rho(\pi_1(\bar{X}))>\ell$ since the only normal subgroup of $U_3(\mathbb{Z}/\ell)$ of order $\ell$ is its center. Therefore, $\rho(\pi_1(\bar{X}))$ is an elementary abelian subgroup of $U_3(\mathbb{Z}/\ell)$ of order $\ell^2$. In particular, $\rho$ is trivial on the subgroup $T$ of $\pi_1(\bar{X})$ generated by commutators and by $\ell^{\mathrm{th}}$ powers, where $\pi_1(\bar{X})/T$ is the maximal elementary abelian $\ell$-quotient group of $\pi_1(\bar{X})$.  This group is isomorphic to $\mathrm{Pic}(\bar{X})[\ell]$.  By our assumption, this latter group is defined over $F$, which means that $\mathrm{Gal}(\bar{F}/F)$ acts trivially on $\pi_1(\bar{X})/T$. This implies that 
$\rho(\pi_1(\bar{X}))$ is an elementary abelian  subgroup of 
$\rho(\pi_1(X)) =  U_3(\mathbb{Z}/\ell)$ of order $\ell^2$, and the quotient group 
$\rho(\pi_1(X))/\rho(\pi_1(\bar{X}))$ is of order $\ell$ and acts trivially on
$\rho(\pi_1(\bar{X}))$.  This forces $\rho(\pi_1(X)) = U_3(\mathbb{Z}/\ell)$ to be abelian, which is not true.  The contradiction completes the proof.
\end{proof}


\section{Restriction of Massey products on curves to the separable closure}
\label{s:restrict}

We make the same assumptions as in the previous section. In other words, $F$ is a field with a fixed separable closure $\bar{F}$ inside a fixed algebraic closure $F^\mathrm{alg}$, and $X$ is a smooth projective geometrically irreducible curve over $F$ such that $\bar{X}= X\otimes_F \bar{F}$ is not isomorphic to $\mathbb{P}^1_{\bar{F}}$. Moreover, $\ell$ is a prime number such that $\ell$ is invertible in $F$. Let $G_F=\mathrm{Gal}(\bar{F}/F)$.

When $F$ is the algebraic closure of a finite field and $\mathbb{Z}/\ell$ is replaced by $\mathbb{Q}_\ell$, Deligne, Griffiths, Morgan and Sullivan  discuss the connection between the Weil conjectures and the vanishing of all higher order Massey products in \cite[p. 246]{DeligneEtAl}. 

We now return to the case with coefficients in $\mathbb{Z}/\ell$ over an arbitrary field $F$.

\begin{proposition}
\label{prop:MasseyOverClosure}
Suppose $t\ge 3$, $\chi_1,\ldots,\chi_t\in \HH^1(X,\mathbb{Z}/\ell)$, and let $\bar{\chi}_1,\ldots,\bar{\chi}_t$ denote their restrictions to $\HH^1(\bar{X},\mathbb{Z}/\ell)$. If the $t$-fold Massey product $\langle\chi_1,\ldots,\chi_t\rangle$ is not empty, then the $t$-fold Massey product $\langle \bar{\chi}_1,\ldots,\bar{\chi}_t\rangle$ is non-empty and contains $0$.
\end{proposition}

\begin{proof}
If $\bar{\chi}_1= 0$, then $\langle\bar{\chi}_1,\bar{\chi}_2,\ldots,\bar{\chi}_t\rangle=\langle 0,\bar{\chi}_2,\ldots,\bar{\chi}_t\rangle$ obviously contains $0$.

Suppose now that $\bar{\chi}_1\ne 0$. By our assumption, we have continuous functions $\kappa_{i,j}:\pi_1(X)\to \mathbb{Z}/\ell$, for $1\le i<j\le t$, $(i,j)\ne (1,t)$, such that there is a continuous group homomorphism
$$\vartheta: \pi_1(X) \to U_{t+1}(\mathbb{Z}/\ell)/Z(U_{t+1}(\mathbb{Z}/\ell))$$ 
as in (\ref{eq:rhobar}).
Then $\vartheta$ restricts to a continuous group homomorphism $\pi_1(\bar{X}) \to U_{t+1}(\mathbb{Z}/\ell)/Z(U_{t+1}(\mathbb{Z}/\ell))$ by restricting the continuous functions $\kappa_{i,j}$ to $\pi_1(\bar{X})$. We denote these latter restrictions by $\bar{\kappa}_{i,j}$. In particular, $\langle \bar{\chi}_1,\ldots, \bar{\chi}_t\rangle$ contains the class in $\HH^2(\bar{X},\mathbb{Z}/\ell)$ of the two-cocycle $\bar{\nu}$ with
$$\bar{\nu}(\sigma,\tau)= -\bar{\chi}_1(\sigma)\bar{\kappa}_{2,t}(\tau)-\bar{\kappa}_{1,2}(\sigma)\bar{\kappa}_{3,t}(\tau) - \cdots - \bar{\kappa}_{1,t-2}(\sigma)\bar{\kappa}_{t-1,t}(\tau)-\bar{\kappa}_{1,t-1}(\sigma)\bar{\chi}_t(\tau)$$
for all $\sigma,\tau\in \pi_1(\bar{X})$. We are free to add to $\bar{\kappa}_{2,t}$ any element of $\HH^1(\bar{X},\mathbb{Z}/\ell)$. Since we have assumed that $\bar{\chi}_1\ne 0$, we can  adjust $\bar{\kappa}_{2,t}$ in this way, using Remark \ref{rem:voodoo}, to make the class of $\bar{\nu}$ trivial in $\HH^2(\bar{X},\mathbb{Z}/\ell)$ since the Weil pairing is non-degenerate and $\HH^2(\bar{X},\mathbb{Z}/\ell)$ is one-dimensional over $\mathbb{Z}/\ell$. Note that in this process we may have to add to $\bar{\kappa}_{2,t}$ an element of $\HH^1(\bar{X},\mathbb{Z}/\ell)$ that is not the restriction of an element of $\HH^1(X,\mathbb{Z}/\ell)$, though if $\HH^1(X,\mathbb{Z}/\ell)\to \HH^1(\bar{X},\mathbb{Z}/\ell)$ is surjective, we can assume we have such a restriction.
\end{proof}

\begin{corollary}
\label{cor:duh}
Let $t\ge 3$ and $\psi_1,\ldots,\psi_t\in \HH^1(\bar{X},\mathbb{Z}/\ell)$. If $\langle\psi_1,\ldots,\psi_t\rangle$ is not empty then it contains $0$.
\end{corollary}

\begin{remark}
\label{rem:Ekedahl}
For surfaces over a separably closed field, the situation is completely different. Ekedahl gave examples in \cite{Ekedahl} of smooth projective surfaces $S$ over $\mathbb{C}$ and characters $\chi_1,\chi_2,\chi_3\in \HH^1(S,\mathbb{Z}/\ell)$ such that $\langle \chi_1,\chi_2,\chi_3\rangle $ does not contain $0$.
\end{remark}

In the situation of Proposition \ref{prop:MasseyOverClosure}, the question arises when the Massey product $\langle \chi_1,\ldots,\chi_t\rangle$ contains zero. This question sometimes reduces to the question of when the $t$-fold Massey product of $t$ characters in $\mathrm{Hom}(G_F,\mathbb{Z}/\ell)$ vanishes. The following definition is useful in this context. 

\begin{definition}
\label{def:Masseyvanish}
We say the $t$-fold Massey vanishing property holds for $F$ over $\mathbb{Z}/\ell$ if 
for all $\alpha_1,\ldots,\alpha_t\in\HH^1(F,\mathbb{Z}/\ell)=\mathrm{Hom}(G_F,\mathbb{Z}/\ell)$, the $t$-fold Massey product $\langle \alpha_1,\ldots,\alpha_t \rangle$ contains zero provided it is non-empty. 
\end{definition}

\begin{remark}
\label{rem:duh!}
Here are some known instances when the $t$-fold Massey vanishing property holds for $F$ over $\mathbb{Z}/\ell$:
\begin{itemize}
\item If $t=3$ and $F$ is arbitrary, this holds by \cite{MinacTan2016}.
\item If $t\ge 4$ and $F$ is a number field, this holds by \cite{HarWit2019}. 
\item  If $F$ is a finite field this holds for all $t \ge 2$ since $\HH^2(F,\mathbb{Z}/\ell) = 0$.
\end{itemize}
\end{remark}

\begin{proposition}
\label{prop:MasseyRestrict}
Let $t\ge 3$. Suppose the $\ell$-torsion $\mathrm{Pic}(\bar{X})[\ell]$ is defined over $F$ and that the $t$-fold Massey vanishing property holds for $F$ over $\mathbb{Z}/\ell$. Let $\chi_1,\ldots,\chi_t\in \HH^1(X,\mathbb{Z}/\ell)$, and suppose the $t$-fold Massey product $\langle\chi_1,\ldots,\chi_t\rangle$ is not empty.  If $\bar{\chi}_{i_0}=0$ for some $1\le i_0\le t$, then $\langle \chi_1,\ldots,\chi_t\rangle$ contains zero.
\end{proposition}

\begin{proof}
Suppose $\bar{\chi}_{i_0}=0$ for some $i_0$. Since $\chi_i\cup \chi_{i+1}=0$ for $i=1,\ldots,t-1$ and since the cup product is anti-commutative, it follows from Lemma \ref{lem:cupp} that either $\chi_{j_0}=0$ for some $j_0$, or all $\bar{\chi}_1=\bar{\chi}_2=\cdots=\bar{\chi}_t=0$. If $\chi_{j_0}=0$ then it is obvious that $\langle \chi_1,\ldots,\chi_t\rangle$ contains zero. Otherwise all of $\chi_1,\ldots,\chi_t$ factor through $\HH^1(F,\mathbb{Z}/\ell)=\mathrm{Hom}(G_F,\mathbb{Z}/\ell)$. Since we assume that the $t$-fold Massey vanishing property holds for $F$ over $\mathbb{Z}/\ell$, it follows that $\langle\chi_1,\ldots,\chi_t\rangle$ contains $0$.
\end{proof}


\section{Necessary conditions for the non-vanishing of Massey products}
\label{s:necessary}

We make the same assumptions as in the previous section. We obtain the following necessary conditions for the $t$-fold Massey product to not contain zero.

\begin{proposition}
\label{prop:nice!}
Suppose that $\ell,t\ge 3$, and that the $\ell$-torsion $\mathrm{Pic}(\bar{X})[\ell]$ is defined over $F$. Moreover, assume that the $t$-fold Massey vanishing property holds for $F$ over $\mathbb{Z}/\ell$. Let $\chi_1,\ldots,\chi_t\in \HH^1(X,\mathbb{Z}/\ell)$ be such that the $t$-fold Massey product $\langle \chi_1,\ldots,\chi_t\rangle$ is not empty and does not contain zero. Then 
the following is true.
\begin{itemize}
\item[(a)] None of the restrictions $\bar{\chi}_1,\ldots,\bar{\chi}_t$ to $\HH^1(\bar{X},\mathbb{Z}/\ell)$ are zero.
\item[(b)] If there exist $a_1,\ldots,a_{t-1} \in (\mathbb{Z}/\ell)^\times$ with $\chi_i=a_i\chi_t$ for all $1\le i \le t-1$
then $t\ge \ell$.
\item[(c)] If $X$ has genus $1$ then there are always $a_1,\ldots,a_{t-1}$ as in $(b)$.
\end{itemize}
\end{proposition}

\begin{proof}
If one of $\chi_1,\ldots,\chi_t$ is zero, then $\langle \chi_1,\ldots,\chi_t\rangle$ contains zero. Therefore, we have that none of $\chi_1,\ldots,\chi_t$ are zero.   Since $t\ge 3$, in order for $\langle \chi_1,\ldots,\chi_t\rangle$ to be not empty, we must have $\chi_i \cup \chi_{i+1} = 0$ for $1 \le i \le t - 1$.  If one of the $\bar{\chi}_i  $ is zero, then by Lemma \ref{lem:cupp} and the anti-commutativity of the cup product we conclude that  $\bar{\chi}_i=0$ for $1\le i\le t$. In other words, $\chi_1,\ldots,\chi_t\in \HH^1(F,\mathbb{Z}/\ell) = \mathrm{Hom}(G_F,\mathbb{Z}/\ell)$. Since we assume that the $t$-fold Massey vanishing property holds for $F$ over $\mathbb{Z}/\ell$, it follows that $\langle \chi_1,\ldots,\chi_t\rangle$ contains zero. This implies that none of $\bar{\chi}_1, \ldots,\bar{\chi}_t$ are zero, which is condition (a). 

Suppose now that there exist $a_1,\ldots,a_{t-1} \in (\mathbb{Z}/\ell)^\times$ as in (b). By (\ref{eq:ohyeah!}), it follows that $\langle \chi_1,\ldots,\chi_t\rangle$ does not contain zero if and only if the $t$-fold Massey product $\langle \chi_t,\ldots,\chi_t \rangle$ of $t$ copies of $\chi_t$ does not contain zero. By (\ref{eq:Bockstein}), $\langle \chi_t\rangle^\ell$ is non-empty, which implies that if $t < \ell$ then $\langle \chi_t\rangle^t$ is non-empty and contains zero. Therefore, the $t$-fold Massey product $\langle \chi_t,\ldots,\chi_t \rangle$ of $t$ copies of $\chi_t$ is non-empty and contains zero if $t<\ell$. Since we have assumed that
$\langle \chi_1,\ldots,\chi_t \rangle$ does not contain $0$, this implies $t \ge \ell$.

Finally, suppose that $X$ has genus 1, so that, by Remark \ref{rem:voodoo}, $\bar{X}$ is an elliptic curve and $\mathrm{Pic}(\bar{X})[\ell]$ has dimension $2$ over $\mathbb{Z}/\ell$. Since $\bar{\chi}_i\cup \bar{\chi}_{i+1}=0$, for $1\le i \le t-1$, and since $\bar{\chi}_1,\ldots,\bar{\chi}_t$ are non-zero by part (a), the non-degeneracy of the Weil pairing and Remark \ref{rem:voodoo} imply that there exist $a_1,\ldots,a_{t-1} \in (\mathbb{Z}/\ell)^\times$ with $\bar{\chi}_i=a_i\bar{\chi}_t$ for all $1\le i \le t-1$. Hence there exist $\psi_1,\ldots, \psi_{t-1}\in \HH^1(F,\mathbb{Z}/\ell) = \mathrm{Hom}(G_F,\mathbb{Z}/\ell)$ such that
$$\chi_i = a_i \chi_t + \psi_i \quad\mbox{for $1\le i\le t-1$.}$$
Since $\chi_{t-1}\cup \chi_t = 0$, and $\chi_t \cup \chi_t = 0$ because $\ell\ge 3$, this implies $\psi_{t-1}\cup\chi_t=0$. But since $\bar{\chi}_t\ne 0$ and $\bar{\psi}_{t-1}=0$, this is by Lemma \ref{lem:cupp} only possible if $\psi_{t-1}=0$. By induction on $t$ we obtain $\psi_i=0$ for all $1\le i\le t-1$. Therefore, there exist $a_1,\ldots,a_{t-1} \in (\mathbb{Z}/\ell)^\times$ with $\chi_i=a_i\chi_t$ for all $1\le i \le t-1$. 
\end{proof}

The next result is an immediate consequence of Proposition \ref{prop:nice!} and the Massey vanishing results in \cite{MinacTan2016} (see Remark \ref{rem:duh!}).

\begin{corollary}
\label{cor:nice!}
Suppose that $\ell\ge 3$, and that the $\ell$-torsion $\mathrm{Pic}(\bar{X})[\ell]$ is defined over $F$. Let $\chi\in \HH^1(X,\mathbb{Z}/\ell)$ be such that the triple Massey product $\langle \chi,\chi,\chi\rangle$ does not contain zero. Then $\ell=3$ and the restriction $\bar{\chi}$ to $\HH^1(\bar{X},\mathbb{Z}/\ell)$ is not zero.
\end{corollary}

\begin{remark}
\label{rem:Kim}
We obtain the following connection to an invariant suggested by M. Kim in \cite{Kim2020}.
Suppose $\ell=3$, and that the $3$-torsion $\mathrm{Pic}(\bar{X})[3]$ is defined over $F$. The non-degeneracy and Galois equivariance of the Weil pairing then imply that $F$ contains $\mu_3(\bar{F})$.  Let $\chi:\pi_1(X)\to\mathbb{Z}/3$ be a character whose restriction $\bar{\chi}$ to $\HH^1(\bar{X},\mathbb{Z}/3)$ is not zero. By (\ref{eq:Bockstein}), the restricted triple Massey product $\langle \chi\rangle^3$ is a singleton
$$\langle \chi\rangle^3= -\beta(\chi)\quad\in\quad \HH^2(X,\mathbb{Z}/3)$$ 
where $\beta$ is the Bockstein operator associated to the exact sequence
$$0\to\mathbb{Z}/3 \to \mathbb{Z}/9\to \mathbb{Z}/3\to 0.$$ 
Since $\langle \chi\rangle^3\subseteq\langle \chi,\chi,\chi\rangle$ and since the cup product is anti-commutative, it follows from Remark \ref{rem:tripleMassey} that
$$\langle \chi,\chi,\chi\rangle = -\beta(\chi) + \chi\cup\HH^1(X,\mathbb{Z}/3).$$
Since $\chi\cup\chi=0$, we obtain that 
\begin{equation}
\label{eq:Kimeq}
\chi\cup\langle \chi,\chi,\chi\rangle = \chi\cup\langle \chi\rangle^3=-\chi\cup\beta(\chi) \in \HH^3(X,\mathbb{Z}/3).\end{equation}
Suppose now that $F$ is a finite field. Then $\HH^3(X,\mu_3)$ is canonically isomorphic to $\mathbb{Z}/3$, so since $F$ contains $\mu_3(\bar{F})$ we get an isomorphism
$\HH^3(X,\mathbb{Z}/3) = \mu_3(\bar{F})^{\otimes -1} = \mathrm{Hom}(\mu_3(\bar{F}),\mathbb{Z}/3)$.  In this case, (\ref{eq:Kimeq}) is  the negative of the invariant Kim defines at the end of  \cite[Section~1]{Kim2020}. 
\end{remark}

\begin{remark}
\label{rem:Kim2}  Continuing with the hypotheses of Remark \ref{rem:Kim}, suppose in addition that 
$X$ has genus $1$ and that $F$ is a finite field.  We claim that we have an isomorphism  of one-dimensional $\mathbb{Z}/3$ vector spaces 
\begin{equation}
\label{eq:nicefact}
\frac{\HH^2(X,\mathbb{Z}/3)}{\chi\cup\HH^1(X,\mathbb{Z}/3)} \to \HH^3(X,\mathbb{Z}/3) \cong \mathbb{Z}/3
\end{equation}
defined by $\beta \mapsto \chi\cup\beta$ for $\beta\in \HH^2(X,\mathbb{Z}/3)$.  Since $F$ is a finite field containing $\mu_3(\bar{F})$,  the cup product 
$$\HH^1(X,\mathbb{Z}/3) \times \HH^2(X,\mathbb{Z}/3) \to \HH^3(X,\mathbb{Z}/3)$$
is non-degenerate, see \cite[Cor.  V.2.3]{Milne}.  Hence (\ref{eq:nicefact}) is well-defined and surjective. 
Since $X$ has genus $1$ and $\bar{\chi}$ is not zero, the argument proving the last statement of Proposition \ref{prop:nice!} shows that the only elements $\xi \in \HH^1(X,\mathbb{Z}/3)$ such that $\chi \cup \xi = 0$ are those $\xi$ that are multiples of $\chi$.  So $\chi\cup\HH^1(X,\mathbb{Z}/3)$ has dimension 
$$\mathrm{dim}_{\mathbb{Z}/3} \HH^1(X,\mathbb{Z}/3) - 1 = \mathrm{dim}_{\mathbb{Z}/3} \HH^2(X,\mathbb{Z}/3) - 1.$$
This proves both sides of (\ref{eq:nicefact}) have dimension $1$, so (\ref{eq:nicefact}) is an isomorphism because it is surjective. The conclusion is that under the above assumptions, the non-triviality of $\langle \chi, \chi, \chi \rangle$ in the group on the left side of (\ref{eq:nicefact}) is equivalent to the non-vanishing of Kim's invariant. We will analyze the $\chi$ for which this holds in the next sections.   
\end{remark}


\section{Triple Massey products and elliptic curves}
\label{s:triple}

In this section, we make the same assumptions as in the previous section. But we focus on the case when $t=3$ and $X=E$ is an elliptic curve over a field $F$ whose characteristic is not $3$.
We fix $\chi_1,\chi_2,\chi_3\in \HH^1(E,\mathbb{Z}/\ell)$ and we assume that the triple Massey product $\langle \chi_1,\chi_2,\chi_3\rangle$ is not empty. By Remark \ref{rem:tripleMassey}, this is equivalent to $\chi_1\cup\chi_2 = \chi_2\cup\chi_3 = 0$. We define $\bar{E}=E\otimes_F\bar{F}$ and $G_F=\mathrm{Gal}(\bar{F}/F)$. Finally, we will slightly abuse notation and denote by $\bar{E}[\ell]$ the set  $E(\bar{F})[\ell]$, endowed with its canonical structure of $G_F$-module.

The next result is an immediate consequence of Proposition \ref{prop:nice!} and the Massey vanishing results in \cite{MinacTan2016} (see Remark \ref{rem:duh!}).

\begin{lemma}
\label{lem:nice!}
Suppose that $E$ is an elliptic curve over a field $F$ of characteristic different from $3$. Assume that $\ell\ge 3$, and that the $\ell$-torsion $\mathrm{Pic}(\bar{E})[\ell]=\bar{E}[\ell]$ is defined over $F$. Let $\chi_1,\chi_2,\chi_3\in \HH^1(E,\mathbb{Z}/\ell)$ be such that the triple Massey product $\langle \chi_1,\chi_2,\chi_3\rangle$ is not empty and does not contain zero. Then $\ell=3$, none of the restrictions $\bar{\chi}_1,\bar{\chi}_2,\bar{\chi_3}$ to $\HH^1(\bar{E},\mathbb{Z}/\ell)$ are zero, and there exist $a,b \in (\mathbb{Z}/\ell)^\times$ with $\chi_2=a\chi_1=b\chi_3$.
\end{lemma}

In particular, it follows from Lemma \ref{lem:nice!} that if $\ell\ge 3$ and the $\ell$-torsion $\bar{E}[\ell]$ is defined over $F$, then $\langle \chi_1,\chi_2,\chi_3\rangle$ contains zero unless possibly when $\ell=3$ and $\chi_1,\chi_2,\chi_3$ all generate the same non-trivial subgroup of $\HH^1(E,\mathbb{Z}/3)$. Using (\ref{eq:ohyeah!}), we see that the only question that remains to be answered is for which characters $\chi: \pi_1(E) \to \mathbb{Z}/3$ of order $3$, the Massey product $\langle \chi,\chi,\chi \rangle$ does not contain zero.

For $n \ge 1$, let $\bar{E}[3^n]$ be the $3^n$-torsion of $\bar{E}$. We assume that $\mathrm{Pic}(\bar{E})[3]=\bar{E}[3]$ is defined over $F$.  Since the Weil pairing is non-degenerate and Galois equivariant, it follows, using Remark \ref{rem:voodoo}, that $F$ contains a primitive cubic root $\zeta_3$ of unity. Our goal is to determine all characters $\chi:\pi_1(E)\to\mathbb{Z}/3$ such that the restriction $\bar{\chi}$ of $\chi$ to $\pi_1(\bar{E})$ is non-zero and the triple Massey product  $\langle \chi,\chi,\chi\rangle$ does not contain zero.

Let $H$ be the subgroup of $U_4(\mathbb{Z}/3)$ defined in Remark \ref{rem:U4!} consisting of all matrices of the form $N=N(a,u,v,w)$ as in (\ref{eq:matrixnotation2}). There is a character $\psi:H \to \mathbb{Z}/3$ sending each such matrix $N$ to $a$.  Similarly to the discussion following (\ref{eq:rhobar}), we see that $\langle\chi,\chi,\chi\rangle$ does not contain zero if and only if there is no continuous group homomorphism $\rho:\pi_1(E) \to H$ with $\chi = \psi \circ \rho$.

The pro-$3$ completion of $\pi_1(\bar{E})$ is isomorphic to the $3$-adic Tate module $T_3(\bar{E}) \cong \mathbb{Z}_3^2$.  Since $\bar{E}[3]=T_3(\bar{E})/3T_3(\bar{E})$ is defined over $F$, each $\sigma\in G_F$ acts on $T_3(\bar{E})$ as the identity modulo $3T_3(\bar{E})$. Let $J^{(3)}$ denote the pro-$3$ completion of a profinite group $J$. 

\begin{lemma}
\label{lem:exact3}
Let $F$ be a field whose characteristic is not $3$, and let $E$ be an elliptic curve over $F$ such that the $3$-torsion $\bar{E}[3]$ is defined over $F$. There is an exact sequence
\begin{equation}
\label{eq:up}
0 \to T_3(\bar{E}) \to \pi_1(E)^{(3)} \to G_F^{(3)} \to 1.
\end{equation}
\end{lemma}

\begin{proof}
It is clear that $\pi_1(E)^{(3)}$ surjects onto $G_F^{(3)}$ by considering the cover $E\otimes_FF' $ of $E$ when $F'$ is the maximal pro-$3$ extension of $F$ in a separable closure of $F(E)$.  Let $L$ be the extension of $F'(E)$ corresponding to the kernel of the resulting homomorphism $\pi_1(E)^{(3)} \to G_F^{(3)}$.  To prove $(\ref{eq:up})$ is exact, it will suffice to show that the natural homomorphism $\omega:\pi_1(\bar{E})^{(3)} = T_3(\bar{E}) \to \mathrm{Gal}(L/F'(E))$ resulting from restricting automorphisms is an isomorphism.  The constant field of $L$ is $F'$ since it is Galois over $F$ and a pro-$3$ extension of $F'$.  So  the base change of $L/F'(E)$
by the extension $\bar{F}/F'$ gives an isomorphism of Galois groups $\mathrm{Gal}(\bar{F}L/\bar{F}(E)) = \mathrm{Gal}(L/F'(E))$.  Here $\mathrm{Gal}(\bar{F}L/\bar{F}(E))$ is a quotient of $T_3(\bar{E})$, and this implies $\omega$ is surjective.   To show $\omega$ is injective, we first claim that all of the $3$-power torsion of $\bar{E}$ is defined over $F'$.  For this, we use the hypothesis that $\bar{E}[3]$ is defined over $F$.  This implies that the action of $G_F$ on $T_3(\bar{E})$ is via matrices that are congruent to the identity mod $3$.  Since the multiplicative group of such matrices is a pro-$3$ group, all the $3$-power torsion of $\bar{E}$ is defined over the maximal pro-$3$ extension $F'$ of $F$.  Now the tower of isogenies over $E\otimes_F F'$ produced by multiplication by powers of $3$ gives
an extension of $F'(E)$ that is Galois over $F(E)$ and has Galois group $T_3(\bar{E})$ over $F'(E)$.  This shows $\omega$ is injective and completes the proof of Lemma \ref{lem:exact3}.
\end{proof}

Let $\mathfrak{G}_0$ be the decomposition group (inside $\pi_1(E)$) of an inverse system of discrete valuations over the origin of $E$ in a cofinal system of finite \'etale covers of $E$.
The sequence (\ref{eq:up}) splits since the image of $\mathfrak{G}_0$ inside $\pi_1(E)^{(3)}$ is isomorphic to $G_F^{(3)}$  and disjoint from the image of $T_3(\bar{E})$ in $\pi_1(E)^{(3)}$. Since $9T_3(\bar{E})$ is a characteristic subgroup of $T_3(\bar{E})$, (\ref{eq:up}) leads to an exact sequence
\begin{equation}
\label{eq:up?}
0 \to \frac{T_3(\bar{E})}{9T_3(\bar{E})} \to \frac{\pi_1(E)^{(3)}}{9T_3(\bar{E})} \to G_F^{(3)} \to 1.
\end{equation}
Let $\sigma\in G_F^{(3)}$. Since $\sigma$ acts on $T_3(\bar{E})$ as the identity modulo $3T_3(\bar{E})$, we have
\begin{equation}
\label{eq:good!}
(\sigma-1)^2(T_3(\bar{E})) \subset 9T_3(\bar{E}).
\end{equation}
Since
$$(\sigma^9-1) \equiv (\sigma-1)^9 + 3\sigma^3(\sigma-1)^3\mod 9\mathbb{Z}[\sigma]$$
we obtain
\begin{equation}
\label{eq:gooder!}
(\sigma^9-1)(T_3(\bar{E})) \subset  (\sigma-1)^3(T_3(\bar{E})) +9T_3(\bar{E})   \subset 9T_3(\bar{E})
\end{equation}
where the second inclusion follows from (\ref{eq:good!}).
In (\ref{eq:up?}), we view $T_3(\bar{E}) / 9T_3(\bar{E})=\bar{E}[9]$ as a (normal) subgroup of $\pi_1(E)^{(3)}/9T_3(\bar{E})$, and we identify $G_F^{(3)}$ with the image of the decomposition group $\mathfrak{G}_0$ inside $\pi_1(E)^{(3)}/9T_3(\bar{E})$. Let $(G_F^{(3)})^9$ be the (normal) subgroup of $G_F^{(3)}$ generated by all $9^{\mathrm{th}}$ powers. By (\ref{eq:gooder!}), the elements of $(G_F^{(3)})^9$ commute with the elements of $\bar{E}[9]$, implying that $(G_F^{(3)})^9$ is a normal subgroup of $\pi_1(E)^{(3)}/9T_3(\bar{E})$ that has trivial intersection with $\bar{E}[9]$. Hence (\ref{eq:up?}) leads to an exact sequence
\begin{equation}
\label{eq:up2}
0 \to \frac{T_3(\bar{E})}{9T_3(\bar{E})} \to \frac{\pi_1(E)^{(3)}/9T_3(\bar{E})}{(G_F^{(3)})^9} \to \frac{G_F^{(3)}}{(G_F^{(3)})^9} \to 1 .
\end{equation}
We define
\begin{equation}
\label{eq:needthis}
\overline{\mathcal{T}}:=\frac{T_3(\bar{E})}{9T_3(\bar{E})} = \bar{E}[9],\quad
\overline{G}:=\frac{\pi_1(E)^{(3)}/9T_3(\bar{E})}{ (G_F^{(3)})^9}\quad\mbox{and}\quad
\overline{G}_F:=\frac{G_F^{(3)}}{ (G_F^{(3)})^9}.
\end{equation}
Letting $\xi:\overline{G}_F\to \mathrm{Aut}(\overline{\mathcal{T}})=\mathrm{GL}_2(\mathbb{Z}/9)$ be the continuous group homomorphism induced by (\ref{eq:up2}), $\overline{G}$ is the semidirect product 
$$\overline{G} = \overline{\mathcal{T}} \rtimes_\xi \overline{G}_F.$$
We view $\overline{\mathcal{T}}$ as a subgroup of $\overline{G}$ and we identify $\overline{G}_F$ with the image of the decomposition group $\mathfrak{G}_0$ inside $\overline{G}$, which has trivial intersection with $\overline{\mathcal{T}}$. 

If $\chi:\pi_1(E)\to \mathbb{Z}/3$ is a non-trivial character, then $\chi$ factors through the maximal elementary abelian $3$-quotient of $\pi_1(E)$, and hence through $\overline{G}$. Since the group $H$ defined in Remark \ref{rem:U4!} has exponent 9, we see that $\langle \chi,\chi,\chi\rangle$ does not contain zero if and only if $\chi$, when viewed as a character from $\overline{G}$ to $\mathbb{Z}/3$, cannot be lifted to a continuous group homomorphism $\rho: \overline{G} \to H$ such that $\chi=\psi\circ \rho$. 

\begin{theorem}
\label{thm:classifiynontrivial!}
Let $F$ be a field whose characteristic is not $3$, and let $E$ be an elliptic curve over $F$ such that the $3$-torsion $\bar{E}[3]$ is defined over $F$. Let $\chi:\pi_1(E)\to\mathbb{Z}/3$ be a character.
Let $\overline{K}_T$ be the image in $\bar{E}[9]$ of the kernel of $\chi$ restricted to $T_3(\bar{E})$, and let $K_F$ be the kernel of $\chi$ restricted to the decomposition group $\mathfrak{G}_0$. Then $\langle \chi,\chi,\chi \rangle$ does not contain zero if and only if 
the restriction $\bar{\chi}$ of $\chi$ to $\pi_1(\bar{E})$ is non-zero and
one of the following two conditions holds:
\begin{itemize}
\item[(1)]
there exist elements $a\in \overline{K}_T-3\bar{E}[9]$ and $\sigma\in G_F^{(3)}$ such that $\sigma(a)\not\in (\mathbb{Z}/9)\, a$, 
or
\item[(2)]
the fixed field of $K_F$ inside $\bar{F}$ is a cubic extension of F that does not contain any primitive ninth root $\zeta$ of unity, and for all $a\in \overline{K}_T-3\bar{E}[9]$ and all $b\in \bar{E}[9]-\overline{K}_T$ and all $\iota\in K_F$ with $\iota(\zeta)=\zeta^4$, we have $(\iota-4)(b)\not\in(\mathbb{Z}/9)\,a$.
\end{itemize}
\end{theorem}

\begin{proof}
We prove Theorem \ref{thm:classifiynontrivial!} by going through all possible cases and showing that $\langle \chi,\chi,\chi\rangle$ contains zero unless the restriction $\bar{\chi}$ of $\chi$ to $\pi_1(\bar{E})$ is non-zero and either condition (1) or condition (2) holds.

If $\bar{\chi}$ is zero, then it follows from the Massey vanishing results in \cite{MinacTan2016} (see Remark \ref{rem:duh!} and Proposition \ref{prop:MasseyRestrict}) that $\langle \chi,\chi,\chi \rangle$ contains zero. For the remainder of the proof, we assume that $\bar{\chi}$ is non-zero.

As noted in the paragraph before the statement of Theorem \ref{thm:classifiynontrivial!}, we can and will replace $\pi_1(E)$ by $\overline{G}$ in our arguments. In particular, we will replace $\bar{\chi}$ by the restriction of $\chi$ to $\overline{\mathcal{T}}=\bar{E}[9]$ and we will identify $\overline{K}_T$ with the kernel of this restriction. We will also replace $K_F$ by the kernel of $\chi$ restricted to $\overline{G}_F$ which we identified with the image of the decomposition group $\mathfrak{G}_0$ inside $\overline{G}$. Moreover, we will replace the statements in conditions (1) and (2) about $G_F^{(3)}$ by the corresponding statements about $\overline{G}_F$ (inside $\overline{G}$). Let $F_\infty\subset \bar{F}$ be such that we can identify $\overline{G}_F=\mathrm{Gal}(F_\infty/F)$. 

Suppose first that condition (1) of Theorem \ref{thm:classifiynontrivial!} holds. This means there exist elements $a\in \overline{K}_T-3\bar{E}[9]$ and $\sigma\in \overline{G}_F$ such that $\sigma(a)\not\in (\mathbb{Z}/9)\, a$. Since $\bar{\chi}\ne 0$, there exists an element $b\in\bar{E}[9]$ such that $\chi(b)=1$. Hence $\{a,b\}$ is a basis of $\bar{E}[9]$ over $\mathbb{Z}/9$, and we can write $\xi(\sigma)$ as a matrix in $\mathrm{GL}_2(\mathbb{Z}/9)$ with respect to this basis. Since $\xi(\sigma)\,a\not\in (\mathbb{Z}/9)\,a$, there exist $\mu_1,\mu_2\in\mathbb{Z}/9$ such that $\mu_2\not\equiv 0$ mod 3 and
\begin{equation}
\label{eq:soneedy}
(\xi(\sigma)-I)\, a = 3\mu_1 \,a + 3\mu_2 \,b.
\end{equation}
We now want to use the elements $a,b,\sigma$ to show that $\chi$ cannot be lifted to a group homomorphism $\rho:\overline{G}\to H$. Suppose to the contrary that such a $\rho$ exists.
 The entries immediately above the main diagonal of $\rho(a)$, $\rho(b)$ and $\rho(\sigma)$ are $0$, $1$ and $\chi(\sigma)$, respectively. This means that there exist $r,s,t,u,v,w,x,y,z\in\mathbb{Z}/3$ such that
$$ \rho(a) =\left(\begin{array}{cccc} 1&0&r&s\\0&1&0&t\\0&0&1&0\\0&0&0&1\end{array}\right),\quad
\rho(b)=\left(\begin{array}{cccc} 1&1&u&v\\0&1&1&w\\0&0&1&1\\0&0&0&1\end{array}\right)
 \quad\mbox{and}\quad
 \rho(\sigma)=\left(\begin{array}{cccc} 1&0&x&y\\0&1&0&z\\0&0&1&0\\0&0&0&1\end{array}\right)
 \rho(b)^{\chi(\sigma)}.$$
 Since $\rho$ is a group homomorphism and $a$ and $b$ commute, we must have that $r=t$ by (\ref{eq:commutator}), which implies that $\rho(a)$ is in the center of $H$. Moreover, $\rho$ must 
 satisfy
 $$[\rho(\sigma),\rho(a)]= \rho(\sigma)\rho(a)\rho(\sigma)^{-1} \rho(a)^{-1} = \rho((\xi(\sigma)-I)\, a).$$  Since $\rho(a)$ is in the center of $H$, this means that $\rho((\xi(\sigma)-I)\, a)$ must be the identity matrix in $H$. However, by (\ref{eq:soneedy}), we obtain that
$$\rho((\xi(\sigma)-I)\, a) = \rho(a)^{3\mu_1} \rho(b)^{ 3\mu_2 } = \left(\begin{array}{cccc}1&0&0&\mu_2\\0&1&0&0\\0&0&1&0\\0&0&0&1\end{array}\right)$$
where the second equation follows from (\ref{eq:powers}). Since $\mu_2\not\equiv 0$ mod 3, this is a contradiction, which means $\rho$ does not exist and $\langle \chi,\chi,\chi\rangle$ does not contain zero.

For the remainder of the proof, we assume that condition (1) does not hold, which means that for all $a\in \overline{K}_T-3\bar{E}[9]$ and all $\sigma\in \overline{G}_F$, we have $\sigma(a)\in (\mathbb{Z}/9)\, a$. Since $\bar{\chi}$ is non-zero, the kernel of $\chi$ restricted to $\bar{E}[9]$ has index 3 in $\bar{E}[9]$. Hence there exists an element $a\in \overline{K}_T-3\bar{E}[9]$. Let $b$ be any element in $\bar{E}[9]-\overline{K}_T$. Then $\{a,b\}$ is a basis of $\bar{E}[9]$ over $\mathbb{Z}/9$ and with respect to this basis, $\xi(\sigma)$ is given by a matrix in $\mathrm{GL}_2(\mathbb{Z}/9)$ of the form
\begin{equation}
\label{eq:matrix1}
\xi(\sigma) = I + 3 \begin{pmatrix} \lambda_1(\sigma) & \mu_1(\sigma)\\0&\mu_2(\sigma)\end{pmatrix}
\quad\mbox{for all $\sigma\in \overline{G}_F$.}
\end{equation}

Suppose first that $K_F=\overline{G}_F$, which means that $\chi(\sigma)=0$ for all $\sigma \in \overline{G}_F$. We will prove that $\langle \chi,\chi,\chi\rangle$ contains zero by constructing a group homomorphism $\rho:\overline{G}\to H$ lifting $\chi$. We define $\rho(a)$ to be the identity in $H$, and 
\begin{equation}
\label{eq:useful1}
\rho(b) =\left(\begin{array}{cccc} 1&\chi(b)&0&0\\0&1&\chi(b)&0\\0&0&1&\chi(b)\\0&0&0&1\end{array}\right)
\quad\mbox{and}\quad
\rho(\sigma)=\left(\begin{array}{cccc} 1&0&0&0\\0&1&0&-\mu_2(\sigma)\\0&0&1&0\\0&0&0&1\end{array}\right)
\quad\mbox{for all $\sigma\in\overline{G}_F$.}
\end{equation}
Then $\rho(a)$ commutes with $\rho(b)$ and $\rho(\sigma)$ for all $\sigma\in\overline{G}_F$. Since $(\xi(\sigma)-I)\,a=3\lambda_1(\sigma) a$, this implies that $\rho$ satisfies the commutator relation $[\rho(\sigma),\rho(a)]=\rho(a)^{3\lambda_1(\sigma)}=\rho((\xi(\sigma)-I)\,a)$  for all $\sigma\in\overline{G}_F$. On the other hand, (\ref{eq:powers}) and (\ref{eq:commutator}) show that for all $\sigma\in\overline{G}_F$, we have
$$[\rho(\sigma),\rho(b)]=\left(\begin{array}{cccc} 1&0&0&\mu_2(\sigma)\chi(b)\\0&1&0&0\\0&0&1&0\\0&0&0&1\end{array}\right) = \rho(a)^{3\mu_1(\sigma)}\rho(b)^{3\mu_2(\sigma)}$$
which implies by (\ref{eq:matrix1}) that $\rho$ satisfies the commutator relation $[\rho(\sigma),\rho(b)]=\rho((\xi(\sigma)-I)\,b)$  for all $\sigma\in\overline{G}_F$. Finally, if $\sigma,\tau\in \overline{G}_F$ then $\xi(\sigma\tau)=\xi(\sigma)\circ\xi(\tau)$, which implies that $\mu_2(\sigma\tau) = \mu_2(\sigma)+\mu_2(\tau)$. Therefore, $\rho(\sigma\tau) = \rho(\sigma)\rho(\tau)$, which shows that $\rho$ is a group homomorphism   lifting $\chi$.

For the remainder of the proof, we assume that $K_F\ne \overline{G}_F$, which means that $K_F=\mathrm{Gal}(F_\infty/N)$ for a degree 3 Galois extension $N/F$. Since the Weil pairing is non-degenerate and Galois equivariant, we obtain, using Remark \ref{rem:voodoo}, that $F$ contains a primitive third root of unity. It follows by Kummer theory that $N=F(\sqrt[3]{\alpha})$ for some $\alpha\in F$. Let $\zeta\in \bar{F}$ be a primitive ninth root of unity. Then $F(\zeta)$ is a cyclic extension of $F$ of degree 1 or 3. In particular, $\zeta\in F_\infty$.

Suppose $\zeta\in N$. Let $\sqrt[9]{\alpha}\in \bar{F}$ be a ninth root of $\alpha$. By Kummer theory, $F(\sqrt[9]{\alpha})$ is a cyclic Galois extension of $F$ of degree 9, so $\sqrt[9]{\alpha}\in F_\infty$. Let $\overline{\tau}$ be the generator of $\mathrm{Gal}(F(\sqrt[9]{\alpha})/F)$ with $\overline{\tau}(\sqrt[9]{\alpha})=\sqrt[9]{\alpha}\,\zeta$. Since $\mathrm{Gal}(F_\infty/F(\sqrt[9]{\alpha}))\subset K_F$, it follows that $\chi$ factors through $\overline{\mathcal{T}} \rtimes_{\overline{\xi}} \mathrm{Gal}(F(\sqrt[9]{\alpha})/F)$ where $\overline{\xi}$ is defined by letting $\overline{\xi}(\overline{\tau})=\xi(\tau)$ when $\tau$ is an extension of $\overline{\tau}$ to $\overline{G}_F$. We will prove that $\langle \chi,\chi,\chi\rangle$ contains zero by constructing a group homomorphism $\rho:\overline{\mathcal{T}} \rtimes_{\overline{\xi}} \mathrm{Gal}(F(\sqrt[9]{\alpha})/F)\to H$ lifting $\chi$.  We define $\rho(a)$ to be the identity in $H$, define $\rho(b)$ as in (\ref{eq:useful1}), and define
\begin{equation}
\label{eq:useful2}
\rho(\overline{\tau})=\left(\begin{array}{cccc} 1&\chi(\tau)&0&0\\0&1&\chi(\tau)&-\mu_2(\tau)\\0&0&1&\chi(\tau)\\0&0&0&1\end{array}\right).
\end{equation}
Then $\rho(a)$ commutes with $\rho(b)$ and $\rho(\overline{\tau})$. We argue as above to see that $\rho$ satisfies the commutator relation $[\rho(\overline{\tau}),\rho(a)]=\rho((\overline{\xi}(\overline{\tau})-I)\,a)$. Moreover, (\ref{eq:powers}) and (\ref{eq:commutator}) show that 
$$[\rho(\overline{\tau}),\rho(b)]=\left(\begin{array}{cccc} 1&0&0&\mu_2(\tau)\chi(b)\\0&1&0&0\\0&0&1&0\\0&0&0&1\end{array}\right) = \rho(a)^{3\mu_1(\tau)}\rho(b)^{3\mu_2(\tau)}$$
which implies by (\ref{eq:matrix1}) that $\rho$ satisfies the commutator relation $[\rho(\overline{\tau}),\rho(b)]=\rho((\overline{\xi}(\overline{\tau})-I)\,b)$. Hence $\rho$ is a group homomorphism  lifting $\chi$.

For the remainder of the proof, we assume that $\zeta\not\in N$. Let $\overline{\iota}$ be the generator of $\mathrm{Gal}(F(\zeta)/F)$ with $\overline{\iota}(\zeta)=\zeta^4$. Let $\sqrt[9]{\alpha}$ be a ninth root of $\alpha$ in $\bar{F}$. Since $F(\sqrt[9]{\alpha},\zeta)$ is a splitting field of $x^9-\alpha$ over $F$, it is Galois. Let $\overline{\iota}_1\in \mathrm{Gal}(F(\sqrt[9]{\alpha},\zeta)/F(\sqrt[9]{\alpha}))$ be such that $\overline{\iota}_1(\zeta)=\zeta^4$, so $\overline{\iota}_1$ extends $\overline{\iota}$, and let $\overline{\tau}\in \mathrm{Gal}(F(\sqrt[9]{\alpha},\zeta)/F(\zeta))$ be such that $\overline{\tau}(\sqrt[9]{\alpha})=\sqrt[9]{\alpha}\,\zeta$. Using Galois theory, we see that $\mathrm{Gal}(F(\sqrt[9]{\alpha},\zeta)/F)$ is generated by $\overline{\iota}_1$ and $\overline{\tau}$ satisfying the relation $\overline{\iota}_1\circ\overline{\tau}\circ\overline{\iota}_1^{-1}=\overline{\tau}^4$. Notice that $\sqrt[9]{\alpha}\in F_\infty$ because $\mathrm{Gal}(F(\sqrt[9]{\alpha},\zeta)/F)$ is a $3$-group. Since $N\subset F(\sqrt[9]{\alpha},\zeta)$, it follows that $\chi$ factors through $\overline{\mathcal{T}} \rtimes_{\overline{\xi}} \mathrm{Gal}(F(\sqrt[9]{\alpha},\zeta)/F)$ where $\overline{\xi}(\overline{\iota}_1)=\xi(\iota_1)$ and $\overline{\xi}(\overline{\tau})=\xi(\tau)$ when $\iota_1$ and $\tau$ are elements in $\overline{G}_F$ that extend $\overline{\iota}_1$ and $\overline{\tau}$, respectively. In particular, $\chi(\iota_1)=0$ and $\chi(\tau)\ne 0$. The three elements of $\mathrm{Gal}(F(\sqrt[9]{\alpha},\zeta)/N)$ that extend $\overline{\iota}$ are $\overline{\iota}_1$, $\overline{\iota}_2$ and $\overline{\iota}_3$, where $\overline{\iota}_1(\sqrt[9]{\alpha})=\sqrt[9]{\alpha}$, $\overline{\iota}_2(\sqrt[9]{\alpha})=\sqrt[9]{\alpha}\,\zeta^3$ and $\overline{\iota}_3(\sqrt[9]{\alpha})=\sqrt[9]{\alpha}\,\zeta^6$. It follows that $\overline{\iota}_2=\overline{\tau}^{-1}\circ\overline{\iota}_1\circ\overline{\tau}$ and $\overline{\iota}_3=\overline{\tau}\circ\overline{\iota}_1\circ\overline{\tau}^{-1}$.

Under the assumptions of the previous paragraph, suppose condition (2) of Theorem \ref{thm:classifiynontrivial!} does not hold. In other words, there exists $a\in \overline{K}_T-3\bar{E}[9]$ and there exists $b\in \bar{E}[9]-\overline{K}_T$ and there exists $\iota\in K_F$ with $\iota(\zeta)=\zeta^4$ such that $(\iota-4)(b)\in(\mathbb{Z}/9)\,a$. In particular, we can use $\{a,b\}$ as the basis with respect to which we write the matrices in (\ref{eq:matrix1}). It follows that $\mu_2(\iota)\equiv 1$ mod 3 in (\ref{eq:matrix1}). Since $\iota$ restricts to one of $\overline{\iota}_1$, $\overline{\iota}_2$ or $\overline{\iota}_3$ in $\mathrm{Gal}(F(\sqrt[9]{\alpha},\zeta)/N)$ and since the latter three elements are conjugate to each other in $\mathrm{Gal}(F(\sqrt[9]{\alpha},\zeta)/F)$, this implies that also $\mu_2(\iota_1)\equiv 1$ mod 3 in (\ref{eq:matrix1}) for any $\iota_1\in K_F$ extending $\overline{\iota}_1$. We will prove that $\langle \chi,\chi,\chi\rangle$ contains zero by constructing a group homomorphism $\rho:\overline{\mathcal{T}} \rtimes_{\overline{\xi}} \mathrm{Gal}(F(\sqrt[9]{\alpha},\zeta)/F)\to H$ lifting $\chi$.  Define $\rho(a)$ to be the identity in $H$, define $\rho(b)$ as in (\ref{eq:useful1}), define $\rho(\overline{\tau})$ as in (\ref{eq:useful2}), and define
$$\rho(\overline{\iota}_1)=\left(\begin{array}{cccc} 1&0&0&0\\0&1&0&-\mu_2(\iota_1)\\0&0&1&0\\0&0&0&1\end{array}\right).$$
Then $\rho(a)$ commutes with $\rho(b)$, $\rho(\overline{\tau})$ and $\rho(\overline{\iota}_1)$. We argue as above to see that $\rho$ satisfies the commutator relations $[\rho(\overline{\sigma}),\rho(a)]=\rho((\overline{\xi}(\overline{\sigma})-I)\,a)$ and $[\rho(\overline{\sigma}),\rho(b)]=\rho((\overline{\xi}(\overline{\sigma})-I)\,b)$ for $\sigma\in\{\iota_1,\tau\}$. It remains to verify the equation $\rho(\overline{\iota}_1)\rho(\overline{\tau})\rho(\overline{\iota}_1^{-1})=\rho(\overline{\tau}^4)$, which is equivalent to the commutator relation $[\rho(\overline{\iota}_1),\rho(\overline{\tau})] = \rho(\overline{\tau})^3$. By (\ref{eq:powers}) and (\ref{eq:commutator}), we obtain
$$[\rho(\overline{\iota}_1),\rho(\overline{\tau})] = \left(\begin{array}{cccc} 1&0&0&\mu_2(\iota_1)\chi(\tau)\\0&1&0&0\\0&0&1&0\\0&0&0&1\end{array}\right)\quad\mbox{and}\quad
\rho(\overline{\tau})^3= \left(\begin{array}{cccc} 1&0&0&\chi(\tau)\\0&1&0&0\\0&0&1&0\\0&0&0&1\end{array}\right).$$
Since $\mu_2(\iota_1)\equiv 1$ mod 3, it follows that $\rho$ is a group homomorphism  lifting $\chi$. 

Finally suppose that for all $a\in \overline{K}_T-3\bar{E}[9]$ and all $b\in \bar{E}[9]-\overline{K}_T$ and all $\iota\in K_F$ with $\iota(\zeta)=\zeta^4$, we have $(\iota-4)(b)\not\in(\mathbb{Z}/9)\,a$. In other words, condition (2) of Theorem \ref{thm:classifiynontrivial!} holds. In particular, it follows that $\mu_2(\iota_1)\not\equiv 1$ mod 3 in (\ref{eq:matrix1}) for any $\iota_1\in K_F$ extending $\overline{\iota}_1$. We want to show that $\chi$ cannot be lifted to a group homomorphism $\rho:\overline{G}\to H$. Suppose to the contrary that such a $\rho$ exists. This means that there exist $r,s,t,u,v,w,x,y,z\in\mathbb{Z}/3$ such that
$$ \rho(b) =\left(\begin{array}{cccc} 1&\chi(b)&r&s\\0&1&\chi(b)&t\\0&0&1&\chi(b)\\0&0&0&1\end{array}\right), \;\;
\rho(\tau)=\left(\begin{array}{cccc} 1&\chi(\tau)&u&v\\0&1&\chi(\tau)&w\\0&0&1&\chi(\tau)\\0&0&0&1\end{array}\right)
 \;\;\mbox{and}\;\;
\rho(\iota_1)=\left(\begin{array}{cccc} 1&0&x&y\\0&1&0&z\\0&0&1&0\\0&0&0&1\end{array}\right).$$
By (\ref{eq:powers}) and (\ref{eq:commutator}), we obtain
$$[\rho(\iota_1),\rho(\tau)] = \left(\begin{array}{cccc} 1&0&0&\chi(\tau)(x-z)\\0&1&0&0\\0&0&1&0\\0&0&0&1\end{array}\right)\quad\mbox{and}\quad
\rho(\tau)^3= \left(\begin{array}{cccc} 1&0&0&\chi(\tau)\\0&1&0&0\\0&0&1&0\\0&0&0&1\end{array}\right).$$
Since $[\iota_1,\tau]=\tau^3$ and $\rho$ is a group homomorphism and since $\chi(\tau)\ne 0$, this forces $x-z\equiv 1$ mod $3$. On the other hand,
$$[\rho(\iota_1),\rho(b)] = \left(\begin{array}{cccc} 1&0&0&\chi(b)(x-z)\\0&1&0&0\\0&0&1&0\\0&0&0&1\end{array}\right)\quad\mbox{and}\quad
\rho(a)^{3\mu_1(\iota_1)}\rho(b)^{3\mu_2(\iota_1)}= \left(\begin{array}{cccc} 1&0&0&\chi(b)\mu_2(\iota_1)\\0&1&0&0\\0&0&1&0\\0&0&0&1\end{array}\right)$$
where the second equality follows since $\chi(a)=0$, which implies that $\rho(a)^3$ is the identiy matrix.
Since $[\iota_1,b]=(\xi(\iota_1)-I)\,b = 3\mu_1(\iota_1)\,a + 3\mu_2(\iota_1)\,b$ and $\rho$ is a group homomorphism and since $\chi(b)\ne 0$, this forces $x-z\equiv \mu_2(\iota_1)$ mod $3$. This is a contradiction, since $\mu_2(\iota_1)\not\equiv 1$ mod 3. Therefore, $\rho$ does not exist, which means that $\langle \chi,\chi,\chi\rangle$ does not contain zero. This completes the proof of Theorem \ref{thm:classifiynontrivial!}. 
\end{proof}

We now proceed to the proof of Theorem \ref{cor:nontrivialgeneral!} from the introduction.

\begin{proof}[Proof of Theorem $\ref{cor:nontrivialgeneral!}$]
If neither condition (i) nor condition (ii) of Theorem \ref{cor:nontrivialgeneral!} are satisfied, it follows from Theorem \ref{thm:classifiynontrivial!} that $\langle \chi,\chi,\chi \rangle$ contains zero for every character $\chi:\pi_1(E)\to\mathbb{Z}/3$. 

Suppose now that either condition (i) or condition (ii) of Theorem \ref{cor:nontrivialgeneral!} holds. As noted in the paragraph before the statement of Theorem \ref{thm:classifiynontrivial!}, we can and will replace $\pi_1(E)$ by $\overline{G}$. As before, let $F_\infty\subset\bar{F}$ be such that we can identify $\overline{G}_F=\mathrm{Gal}(F_\infty/F)$.

If condition (i) holds then there exist $a\in\bar{E}[9]$ and $\sigma\in \overline{G}_F$ such that $\sigma(a)\not\in (\mathbb{Z}/9)\,a$. In particular, $a\not\in 3\bar{E}[9]=\bar{E}[3]$. Let $b\in\bar{E}[9]$ be another element such that $\{a,b\}$ is a basis of $\bar{E}[9]$ over $\mathbb{Z}/9$. We can define a character $\chi:\overline{G}\to\mathbb{Z}/3$ by $\chi(a)=0$, $\chi(b)=1$ and $\chi(\sigma)=0$ for all $\sigma\in\overline{G}_F$. By condition (1) of Theorem \ref{thm:classifiynontrivial!}, it follows that $\langle \chi,\chi,\chi\rangle$ does not contain zero.

If condition (ii) holds, then $L:=F(\zeta)$ is not the only cubic extension of $F$ inside $\bar{F}$. By Kummer theory, there exists $\alpha\in F$ and a cube root $\sqrt[3]{\alpha}$ of $\alpha$ in $\bar{F}$ that does not lie in $F$ such that $F(\sqrt[3]{\alpha})\ne L$. Moreover, condition (ii) implies that for each $\sigma\in\overline{G}_F$, there exists a character $h: \overline{G}_F\to \mathbb{Z}/9$ such that $\sigma$ acts on $\bar{E}[9]$ as multiplication by a scalar of the form $\xi(\sigma)=1-3 h(\sigma)$. By the non-degeneracy of the Weil pairing, there exist two points $P_1,P_2\in \bar{E}[9]$ such that the Weil pairing $\langle P_1,P_2\rangle_{\mathrm{Weil}}=\zeta$. Since the Weil pairing is Galois equivariant and bilinear, we obtain that 
\begin{equation}
\label{eq:sigmazeta}
\sigma(\zeta) = \langle \sigma(P_1),\sigma(P_2)\rangle_{\mathrm{Weil}} = \langle (1-3h(\sigma))P_1,(1-3h(\sigma))P_2\rangle_{\mathrm{Weil}}=\zeta^{(1-3h(\sigma))^2} = \zeta^{1+3h(\sigma)}.
\end{equation}
Define a map $\chi:\overline{G}\to \mathbb{Z}/3$ by letting the restriction of $\chi$ to $\bar{E}[9]$ be a fixed non-zero character, and by letting $\chi(\sigma)=i(\sigma)$ for all $\sigma\in\overline{G}_F$ with $\sigma(\sqrt[3]{\alpha}) = \sqrt[3]{\alpha}\,\zeta^{3i(\sigma)}$. Then $\chi$ is a character of $\overline{G}$ since $\chi([\sigma,c]) = \chi((\xi(\sigma)-1)\,c) = \chi(-3h(\sigma)\,c) = 0$ for all $\sigma\in\overline{G}_F$ and $c\in\bar{E}[9]$. 
The kernel $K_F$ of $\chi$ restricted to $\overline{G}_F$ consists of all $\sigma\in\overline{G}_F$ that fix $\sqrt[3]{\alpha}$. Hence the fixed field of $K_F$ is the cubic extension $F(\sqrt[3]{\alpha})$ of $F$, which does not contain any primitive ninth root $\zeta$ of unity since $F(\sqrt[3]{\alpha})\ne L$. We have that $\mathrm{Gal}(L/F)=\langle \overline{\iota} \rangle$ where $\overline{\iota}(\zeta)=\zeta^4$. If $\iota\in\overline{G}_F$ is any extension of $\overline{\iota}$, we obtain from (\ref{eq:sigmazeta}) that $\iota(\zeta) = \zeta^{1+3h(\iota)}$. Since $\iota(\zeta) =\zeta^4$, this implies that $h(\iota)\equiv1$ mod 3. Hence $\xi(\iota) = 1-3h(\iota) \equiv 7$ mod 9. It follows that $(\iota-4)(c)=(\xi(\iota) - 4)c = 3c$ for all $c\in\bar{E}[9]$. If $a\in \overline{K}_T-3\bar{E}[9]$ and $b\in\bar{E}[9]-\overline{K}_T$ then $\{a,b\}$ is a basis for $\bar{E}[9]$ over $\mathbb{Z}/9$, so  $3b\not\in(\mathbb{Z}/9)\,a$. Hence $(\iota-4)(b)=3b\not\in(\mathbb{Z}/9)\,a$ for all such $a$ and $b$. By condition (2) of Theorem \ref{thm:classifiynontrivial!}, it follows that $\langle \chi,\chi,\chi\rangle$ does not contain zero.
\end{proof}


\section{Non-vanishing triple Massey products for elliptic curves over number fields}
\label{s:ellipticnumber}

Conditions (i) and (ii) of Theorem \ref{cor:nontrivialgeneral!} depend on information concerning the action of $G_F$ on the $9$-torsion of an elliptic curve $E$ defined over $F$.  In this section we analyze two situations in which one has more control on this action.  The first arises from specializations of results of Igusa on Galois actions of generic elliptic curves.  The second arises from the Shimura reciprocity law for CM elliptic curves over number fields.

\begin{example}
\label{ex:igusa}  We first construct $E$ and $F$ for which condition (i) of Theorem \ref{cor:nontrivialgeneral!} is satisfied.  
Let $t$ be an indeterminate, and let $E_t$ be the elliptic curve defined over the field $\Q(t)$ by the equation
\begin{equation}
\label{eq:igusacurve}
y^2=4x^3 - \frac{27t}{t-1728}x - \frac{27t}{t-1728}.
\end{equation}
This curve $E_t$, which was considered by Igusa in \cite{Igusa1959}, has $j$-invariant $t$.

Given an integer $n$, we denote by $\Q(\bar{E}_t[n])$ the field obtained from $\Q(t)$ by adjoining the coordinates of the $n$-torsion points of $\bar{E}_t$ to $\Q(t)$. According to Igusa \cite[Theorem~3]{Igusa1959}, the Galois representation
$$\Gal(\Q(\bar{E}_t[n])/\Q(t)) \to \GL_2(\Z/n)$$
is surjective (hence bijective). It is a well-known fact that the determinant of this representation is the cyclotomic character. Therefore, over the field $k:=\Q(\zeta_n,t)$, the Galois representation
$$\Gal(k(\bar{E}_t[n])/k) \to \GL_2(\Z/n)$$
has image equal to $\SL_2(\Z/n)$.  If we set $n=9$, then by Galois theory we have an exact sequence
$$0 \to \Gal(k(\bar{E}_t[9])/k(\bar{E}_t[3])) \to \SL_2(\Z/9) \to \SL_2(\Z/3) \to 0$$
(here we use Igusa's result twice: for $n=9$ and for $n=3$).
According to Hilbert's irreducibility theorem, these Galois groups remain the same for infinitely many rational specializations $t_0$ of the parameter $t$. Therefore, one obtains infinitely many (non-isomorphic) elliptic curves $E_{t_0}$ over $\Q(\zeta_{9})$ such that
$$\#\Gal(\Q(\zeta_{9}, \bar{E}_{t_0}[9])/\Q(\zeta_{9}, \bar{E}_{t_0}[3])) = \frac{\#\SL_2(\Z/9)}{\#\SL_2(\Z/3)} = 27.$$
If we let $F_{t_0}=\Q(\zeta_{9}, \bar{E}_{t_0}[3])$ then the curve $E_{t_0}$ over the field $F_{t_0}$ satisfies condition (i) of Theorem \ref{cor:nontrivialgeneral!}, since $\#(\mathbb{Z}/9)^\times=6$ is strictly smaller than $27$.
\end{example}

\begin{example}
\label{ex:littleexample}
Suppose $F$ is a number field containing $\mathbb{Q}(\sqrt{3},\sqrt{-1})$ that does not contain a primitive ninth root of unity, and let $E$ be the elliptic curve with model $y^2 = x^3 - 1$. Then $\bar{E}[3]$ is defined over $F$. Since $F$ does not contain a primitive ninth root of unity, it follows from Theorem \ref{cor:nontrivialgeneral!} that there exists an element $\chi \in \HH^1(E,\mathbb{Z}/3)$ with non-vanishing triple Massey product.
\end{example}

Note that in Example \ref{ex:littleexample} we do not determine which of the conditions (i) or (ii) of Theorem \ref{cor:nontrivialgeneral!} holds.  Distinguishing which of these holds involves controlling the action of $G_F$ on  $\bar{E}[9]$.  It is natural then to consider CM elliptic curves and to analyze the information about this Galois action that is provided by the Shimura reciprocity law.

\begin{hypo}
\label{hyp:FraukeIsToBlame}
Let $K$ be an imaginary quadratic field.  Fix an embedding of  $K$ into $\mathbb{C}$.  
Let $\mathcal{O}$ be an order in $K$.  Suppose $\mathcal{A}$ is a non-zero finitely generated $\mathcal{O}$-submodule of $K$.
Fix an isomorphism $\xi:\mathbb{C}/\mathcal{A} \to E$, where $E$ is an elliptic curve  over $\mathbb{C}$ with CM by  $\mathcal{O}$ and this isomorphism is equivariant for the action of $\mathcal{O}$.   Let $L$ be the abelian extension of $K$ that is the ring class field of $\mathcal{O}$.  For $r = 3, 9$ define $F_r$ to be the extension of $L$ obtained by adjoining the coordinates of the $r$-torsion points of $\bar{E}$. 
\end{hypo}  

\begin{theorem}
\label{thm:CM} 
Under Hypothesis $\ref{hyp:FraukeIsToBlame}$, suppose $\mathbb{Z}_3 \otimes_{\mathbb{Z}} \mathcal{A}$ is a free module over $\mathbb{Z}_3 \otimes_{\mathbb{Z}} \mathcal{O}$, which is the case if $\mathbb{Z}_3 \otimes_{\mathbb{Z}} \mathcal{O}$ is \'etale over $\mathbb{Z}_3$.  Then:
\begin{itemize}
\item[(a)]The curve $E$ over the field $F_3$ satisfies condition $(i)$ of Theorem $\ref{cor:nontrivialgeneral!}$.
\item[(b)] There is a field $N$ such that $F_3 \subset N \subset F_9$ and $E$ satisfies condition $(ii)$ of Theorem $\ref{cor:nontrivialgeneral!}$  over $N$.
\end{itemize} 
\end{theorem}  

\begin{remark}  
\label{rem:CM1}
The elliptic curve $\mathbb{C}/\mathcal{O}$ is isogenous to $E = \mathbb{C}/\mathcal{A}$
and  $\mathbb{Z}_3 \otimes_{\mathbb{Z}} \mathcal{O}$ is clearly free over $\mathbb{Z}_3 \otimes_{\mathbb{Z}} \mathcal{O}$.  So we can always replace $E = \mathbb{C}/\mathcal{A}$
by an isogenous elliptic curve $\mathbb{C}/\mathcal{O}$ to which the conclusions (a) and (b) of Theorem \ref{thm:CM} apply.
\end{remark}

\begin{proof}[Proof of Theorem $\ref{thm:CM}$.]
The Shimura reciprocity law \cite[Theorem 5.4]{Shimura1994} has the following consequence.
Let  $s$ be an element of the id\`ele group $J_K$ of $K$. Let $K^{\mathrm{ab}}$ be the maximal abelian extension of $K$, and let $\sigma$ be an extension to $\mathbb{C}$ of the element of $\mathrm{Gal}(K^{\mathrm{ab}}/K)$ which is the image of $s$ under the Artin map.  The ring class field $L$ is by definition the class field associated to the subgroup $K^\times \cdot (\prod_v \mathcal{O}_v^\times \times K_\infty^\times)$ of the id\`eles of $K$, where  $v$ runs over the finite places in $K$ and $K_\infty = \mathbb{C}$ is the completion of $K$ at the unique infinite place.  Suppose $s \in  \prod_v \mathcal{O}_v^\times \times K_\infty^\times$, so that  $\sigma$ fixes $L$ and $s^{-1} \mathcal{A} = \mathcal{A}$. By \cite[Theorem 5.7]{Shimura1994}, $E$ is defined over $L$.  So the twist $E^\sigma$ of $E$ by  $\sigma$ is isomorphic to $E$.  The Shimura reciprocity law therefore shows that there is an automorphism $\lambda(\sigma) \in \mathrm{Aut}(E) = \mathcal{O}^\times$ such that the following diagram commutes:
$$\xymatrix{K/\mathcal{A}\ar[d]_{s^{-1}} \ar[rr]^{\xi} && E\ar[d]^{\sigma}\\ 
K/\mathcal{A} \ar[rr]^{\lambda(\sigma) \circ \xi} && E&.}$$
Here $\xi$ is equivariant with respect to the action of $\mathcal{O}$, so $\lambda(\sigma) \circ \xi = \xi \circ \lambda(\sigma)$ and we can write this diagram as
\begin{equation}
\label{eq:tricky}
\xymatrix{K/\mathcal{A}\ar[d]_{\lambda(\sigma) \cdot s^{-1}} \ar[rr]^{\xi} && E\ar[d]^{\sigma}\\ 
K/\mathcal{A} \ar[rr]^{\xi} && E&.}
\end{equation}
We first use this to bound the extension $F_3$ of $L$ generated by the coordinates of the $3$-torsion points of $\bar{E}$.  Define $\mathcal{O}_3 = \mathbb{Z}_3 \otimes_\mathbb{Z} \mathcal{O} \subset \prod_{v|3\mathcal{O}_K} \mathcal{O}_v$.  Multiplication by elements $s_3 \in (1 + 3\mathcal{O}_3)^\times$ fixes the $3$-torsion $3^{-1} \mathcal{A}/\mathcal{A}$ in $K/\mathcal{A}$.  Let $L_1$ be the abelian extension of $K$ which is the class field to $K^\times \cdot (\prod_{v \nmid 3\mathcal{O}_K} \mathcal{O}_v^\times \times (1 + 3\mathcal{O}_3)^\times \times K_{\infty}^\times)$, so that $L \subset L_1$.  The above diagram for $s \in \prod_{v \nmid 3\mathcal{O}_K} \mathcal{O}_v^\times \times (1 + 3\mathcal{O}_3)^\times \times K_{\infty}^\times$ gives $\sigma \in \mathrm{Gal}(\mathbb{C}/L_1)$ and a commutative square
\begin{equation}
\label{eq:nicer}
\xymatrix{3^{-1}\mathcal{A}/\mathcal{A}\ar[d]_{\lambda(\sigma) } \ar[rr]^{\xi} && E\ar[d]^{\sigma}\\ 
3^{-1}\mathcal{A}/\mathcal{A} \ar[rr]^{\xi} && E&.}
\end{equation}
By hypothesis, $\mathbb{Z}_3 \otimes_{\mathbb{Z}} \mathcal{A}$ is a free rank one $\mathcal{O}_3$-module, so $3^{-1} \mathcal{A}/\mathcal{A}$ is isomorphic to $\mathcal{O}_3/3 \mathcal{O}_3 = \mathcal{O}/3 \mathcal{O}$.  If multiplication by $\lambda(\sigma) \in \mathcal{O}^\times$ is trivial on $3^{-1}\mathcal{A}/\mathcal{A}$, it follows that $\lambda(\sigma) -1 \in 3 \mathcal{O} \subset 3 \mathcal{O}_K$.  However, $\lambda(\sigma)$ is a root of unity of order dividing $4$ or $6$, and the only such root of unity for which $\mathrm{Norm}_{K/\mathbb{Q}}(\lambda(\sigma) - 1)$ is divisible by $9$ is $\lambda(\sigma) = 1$. It now follows from (\ref{eq:nicer}) that the map $\sigma \mapsto \lambda(\sigma)$ must be a homomorphism from $\mathrm{Aut}(\mathbb{C}/L_1)$ to the cyclic group $\mathcal{O}^\times$.  Let $L_2$ be the cyclic extension of $L_1$ which is the fixed field of the kernel of this homomorphism.  

Assume first that $\mathcal{O}^\times$ has an element of order $3$.
Then $\mathcal{O} = \mathbb{Z}[\zeta_3]$ and $K = \mathbb{Q}(\zeta_3) = L$ when $\zeta_3$ is a primitive cube root of unity. The elliptic curve $E$ must be isomorphic to $y^2 = x^3 - 1$, since $\mathcal{O}$ has class number $1$.  The $3$-torsion points of $\bar{E}$ then consist of the point at infinity together with the points with $(x,y)$-coordinates given by elements of $ \{(0,\pm \sqrt{-1}), (4^{1/3},\pm \sqrt{3}), (\zeta_3 4^{1/3},\pm \sqrt{3}), (\zeta_3^2 4^{1/3},\pm \sqrt{3})\}$.  Thus $F_3 = K(\sqrt{-1}, 4^{1/3})$ is cyclic of degree $6$ over $K$, totally ramified over the prime $2 \mathcal{O}_K$ and unramified over all other primes of $\mathcal{O}_K$.  Accordingly, there is a subgroup $T$ of index $6$ in the units $\mathcal{O}_{K,2}^\times = \mathcal{O}_2^\times$ of the completion $\mathcal{O}_{K,2}$ of $K$ at $2 \mathcal{O}_K$ such that the group $U = K^\times \cdot (T \times \prod_{v \nmid 2\mathcal{O}_K} \mathcal{O}_v^\times \times K_\infty^\times)$ has trivial image under the Artin map to $\mathrm{Gal}(F_3/K)$.  We now let $s_3$ be an element of $(1 + 3\mathcal{O}_3)^\times \subset \mathcal{O}_3^\times$, and we let $s$ be the id\`ele with component $s_3$ above $3$ and trivial components at all other places.  Then $s \in U$, since the component of $s$ at the place over $2$ is $1$.  Hence the automorphism $[s,K] \in \mathrm{Gal}(K^{\mathrm{ab}}/K)$ fixes $F_3$ as well as $L_1$.  Therefore if $\sigma$ is any extension of  $[s,K]$ to $\mathrm{Aut}(\mathbb{C}/F_3L_1)$ we find that $\lambda(\sigma)$ is the identity.  Hence (\ref{eq:tricky}) gives a commutative diagram
\begin{equation}
\label{eq:urk}
\xymatrix{9^{-1}\mathcal{A}/\mathcal{A}\ar[d]_{s^{-1}} \ar[rr]^{\xi} && E\ar[d]^{\sigma}\\ 
9^{-1} \mathcal{A}/\mathcal{A} \ar[rr]^{\xi} && E&.}
\end{equation}
Since $\mathcal{O}_3$ is a discrete valuation ring in this case, $\mathcal{A}_3 = \mathbb{Z}_3 \otimes_{\mathbb{Z}} \mathcal{A}$ is automatically free of rank $1$ over $\mathcal{O}_3$.  Hence multiplication by the elements of $(1 + 3\mathcal{O}_3)^\times$ produces $9$ distinct endomorphisms of $9^{-1} \mathcal{A}/\mathcal{A}$.  Thus (\ref{eq:urk}) shows that the action of $\mathrm{Gal}(K^{\mathrm{ab}}/F_3L_1)$ on the $9$-torsion of $9^{-1}\mathcal{A}/\mathcal{A}$ has image a group of order at least $9$. Here $\mathrm{Gal}(K^{\mathrm{ab}}/F_3L_1)$ fixes the $3$-torsion $3^{-1}\mathcal{A}/\mathcal{A}$.  On picking generators for the $9$-torsion, we get a map from $\mathrm{Gal}(K^{\mathrm{ab}}/F_3L_1)$ into the kernel of the reduction map $\mathrm{GL}_2(\mathbb{Z}/9) \to \mathrm{GL}_2(\mathbb{Z}/3)$ whose image has order at least $9$.  Thus this image cannot just consist of scalar matrices, so the curve $E$ over the field $F_3$ satisfies condition (i) of Theorem \ref{cor:nontrivialgeneral!} when $\mathcal{O}^\times$ has order divisible by $3$. To produce an $N$ as in part (b) of Theorem \ref{thm:CM}, let $N$ be the class field associated to the subgroup $U' = K^\times \cdot (T \times U'_3 \times \prod_{v \nmid 6\mathcal{O}_K} \mathcal{O}_v^\times \times K_\infty^\times)$ where $U'_3$ is the subgroup of elements of $\mathcal{O}_3^\times = \mathcal{O}_{K,3}^\times $ that are congruent to elements of $1 + 3 \mathbb{Z}_3$ mod $9$.  Using the same arguments as above, the elements of $\mathrm{Gal}(K^{\mathrm{ab}}/N)$ act on $\bar{E}[9]$ by multiplication by elements of $1 + 3 \mathbb{Z}_3$.  Since $(1 + 3) ^2 \not \equiv 1$ mod $9$, the Weil pairing shows $\mathrm{Gal}(K^{\mathrm{ab}}/N)$ acts non-trivially on the ninth roots of unity, so the curve $E$ over $N$ satisfies condition (ii) of Theorem \ref{cor:nontrivialgeneral!}.

It is interesting to note that in this case, there are elements $s_3$ of $\mathcal{O}_3^\times$ so that if $s$ is the id\`ele with  component $s_3$ above $3$ and trivial components at all other places, the action of $s_3^{-1}$ on $3^{-1} \mathcal{A}/\mathcal{A}$ is of order $6$ but the Artin automorphism $[s,K]$ fixes $F_3$ since $F_3/K$ is unramified above $3$.  Thus when $\sigma \in \mathrm{Aut}(\mathbb{C}/K)$ extends $[s,K]$, the value of $\lambda(\sigma) \in \mathcal{O}^\times$  in diagram (\ref{eq:tricky})  must be a sixth root of unity.

We may now suppose that $\# (\mathcal{O}^\times)$ divides $4$.   Then $\lambda(\sigma)^4$ is the identity.  We conclude that for $s_3 \in (\mathcal{O}_3^\times)^4$ and $s$ the id\`ele with component $s_3$ above $3$ and component $1$ at all other places, the  diagram (\ref{eq:tricky}) becomes
\begin{equation}
\label{eq:better}
\xymatrix{K/\mathcal{A}\ar[d]_{s^{-1}} \ar[rr]^{\xi} && E\ar[d]^{\sigma}\\ 
K/\mathcal{A} \ar[rr]^{ \xi} && E&.}
\end{equation}
Since $\mathcal{A}$ is an $\mathcal{O}$-module, the subgroup $(1 + 3 \mathcal{O}_3)^\times$ of $\mathcal{O}_3^\times$ acts trivially by multiplication on the $3$-torsion $3^{-1}\mathcal{A}/\mathcal{A}$, while $(1 + 9 \mathcal{O}_3)^\times$ acts trivially on $9^{-1} \mathcal{A}/\mathcal{A}$.  Here $(1 + 3\mathcal{O}_3)^\times \subset (\mathcal{O}_3^\times)^4$ since $(1 + 3\mathcal{O}_3)^\times$ is a pro-$3$ group, so   (\ref{eq:better}) shows $\sigma$ acts trivially on the $3$-torsion of $\bar{E}$ if $s_3 \in (1 + 3 \mathcal{O}_3)^\times$.  Thus such $\sigma$ lie in $ \mathrm{Aut}(\mathbb{C}/F_3)$ because $F_3$ is the extension of $L$ obtained by adjoining the coordinates of the $3$-torsion points of $\bar{E}$.  

We now use the hypothesis that $\mathbb{Z}_3 \otimes \mathcal{A}$ is a free rank one $\mathcal{\mathcal{O}}_3$-module to be able to say that $9^{-1} \mathcal{A}/\mathcal{A}$ is a free rank one module for $\mathcal{O}_3/9 \mathcal{O}_3$. This implies that the multiplication by the $9$ elements of $(1 + 3 \mathcal{O}_3)^\times/(1 + 9 \mathcal{O}_3)^\times$ give distinct automorphisms of $9^{-1}\mathcal{A}/\mathcal{A}$, each of which fix $3^{-1}\mathcal{A}/\mathcal{A}$ elementwise. The diagram (\ref{eq:better}) together with $(1 + 3\mathcal{O}_3)^\times \subset (\mathcal{O}_3^\times)^4$  now shows that the elements of $(1 + 3 \mathcal{O}_3)^\times/(1 + 9 \mathcal{O}_3)^\times$ give $9$ distinct automorphisms of the field $F_9$ obtained from $F_3$ by adjoining the coordinates of the $9$-torsion points of $\bar{E}$.  Each of these automorphisms fixes $F_3$, so we have shown $\mathrm{Gal}(F_9/F_3)$ has order at least $9$.  We now argue as in the case when $\mathcal{O}^\times$ has order divisible by $3$ that the curve $E$ over the field $F_3$ satisfies condition (i) of Theorem \ref{cor:nontrivialgeneral!}, and that there is a field $N$ as in part (b) of Theorem \ref{thm:CM}.  This completes the proof.  
\end{proof}

\begin{example}
\label{ex:CMexample}
Let $E$ be the modular curve $X_0(32)$ (which is the strong Weil curve 32A1(B) in Cremona's notation \cite{Cremona}).  By \cite{Elkies}, $E$ has complex multiplication by $\mathbb{Z}[i]$ with multiplication by $i$ arising from the map $z \mapsto z + \frac{1}{4}$ on the upper half plane, which normalizes $\Gamma_0(32)$.  The four rational points are all cusps (and there are four cusps not defined over $\mathbb{Q}$).  Note that there is an isomorphism of $E$ with the curve $y^2 = x^4 - 1$, which is the quotient of the Fermat quartic by an involution.  The conductor of $E$ is $32$ and the complex multiplication of $E$ by $\mathbb{Z}[i]$ is defined over $K = \mathbb{Q}(i)$.  Since $32$ is prime to $3$, $E$ has good reduction above $3$ and the hypotheses of Theorem \ref{thm:CM} are satisfied.
\end{example}

\begin{remark} 
\label{rem:Cebotarev}
Suppose $E$ and $F_3$ are as in Theorem \ref{thm:CM}.  One can show by an easy Cebotarev argument that there are infinitely many prime ideals $\mathfrak{p}$ of $\mathcal{O}_{F_3}$ such that the reduction of $E$ at $\mathfrak{p}$ satisfies condition (i) of Theorem \ref{cor:nontrivialgeneral!} over the residue field of $\mathfrak{p}$.  
\end{remark}


\section{Triple Massey products and elliptic curves over finite fields}
\label{s:ellipticfinite}

In this section, we assume $\ell \ge 3$ and that $E$ is an elliptic curve over a finite field $F=\mathbb{F}_q$ such that $q$ is not divisible by $\ell$. In particular, $G_{\mathbb{F}_q}$ is profinitely generated by a Frobenius automorphism $\Phi$, which we write as $G_{\mathbb{F}_q}=\widehat{\langle \Phi\rangle}$. 

Our goal is to classify all characters $\chi_1, \chi_2, \chi_3:\pi_1(E)\to\mathbb{Z}/\ell$ such that the triple Massey product $\langle \chi_1, \chi_2, \chi_3 \rangle$ does not contain zero. 
If $\bar{E}[\ell]$ is defined over $\mathbb{F}_q$, a complete answer is given by Lemma \ref{lem:nice!} and Theorem \ref{thm:classifiynontrivial!}. Since $\mathbb{F}_q$ is finite, condition (2) of Theorem \ref{thm:classifiynontrivial!} never holds, which simplifies the statement; see Theorem \ref{thm:ellfin3case1} below. Additionally, we will analyze all cases when $\bar{E}[\ell]$ is not defined over $\mathbb{F}_q$. We will see that in these cases $\ell>3$ is possible.

Recall from Remark \ref{rem:U4!} that if $\ell>3$ then every element of $U_4(\mathbb{Z}/\ell)$ has order $1$ or $\ell$. On the other hand, $U_4(\mathbb{Z}/3)$ contains elements of order 9. Define
\begin{equation}
\label{eq:ell'}
\ell':=\left\{\begin{array}{cl}9&\mbox{ if $\ell=3$, and}\\
\ell &\mbox{ if $\ell>3$.}
\end{array}\right.
\end{equation}

For $\ell \ge 3$, we have a short exact sequence
$$0 \to \pi_1(\bar{E}) \to \pi_1(E) \to \widehat{\langle \Phi \rangle} \to 1.$$
Denote by $\mathbf{P}$ the set of all positive rational primes. Defining 
\begin{equation}
\label{eq:Gamma}
\Gamma:=\frac{\pi_1(E)}{\prod_{p\in \mathbf{P}-\{\ell\}}T_p(\bar{E})}
\end{equation}
we obtain an exact sequence
\begin{equation}
\label{eq:upfinfirst}
0 \to T_\ell(\bar{E}) \to \Gamma \to \widehat{\langle \Phi \rangle} \to 1.
\end{equation}
As before, let $\mathfrak{G}_0$ be the decomposition group (inside $\pi_1(E)$) of an inverse system of discrete valuations over the origin of $E$ in a cofinal system of finite \'etale covers of $E$.
The sequence \eqref{eq:upfinfirst} splits since the image of $\mathfrak{G}_0$ inside $\Gamma$ is isomorphic to $\widehat{\langle \Phi\rangle}$ and disjoint from the image of $T_\ell(\bar{E})$ inside $\Gamma$.

With $\ell'$ as in (\ref{eq:ell'}), we obtain an exact sequence
\begin{equation}
\label{eq:upfin0}
0 \to \frac{T_\ell(\bar{E})}{\ell'\,T_\ell(\bar{E})} \to \frac{\Gamma}{\ell'\,T_\ell(\bar{E})} \to \widehat{\langle \Phi \rangle} \to 1.
\end{equation}
We view $T_\ell(\bar{E}) / \ell'\,T_\ell(\bar{E})=\bar{E}[\ell']$ as a (normal) subgroup of $\Gamma/\ell'\,T_\ell(\bar{E})$, and we identify $\widehat{\langle \Phi \rangle}$ with the image of $\mathfrak{G}_0$ inside $\Gamma/\ell'\,T_\ell(\bar{E})$. Let $\widehat{\langle \Phi^{\ell'}\rangle}$ be the subgroup of $\widehat{\langle \Phi \rangle}$ that is profinitely generated by $\Phi^{\ell'}$. By considering the action of $\widehat{\langle \Phi \rangle}$ on $\Gamma/\ell' \,T_\ell(\bar{E})$, we see that the minimal normal subgroup  of $\Gamma/\ell'\,T_\ell(\bar{E})$ that contains $\widehat{\langle \Phi^{\ell'}\rangle}$ is profinitely generated by $\Phi^{\ell'}$ together with $(\Phi^{\ell'}-1)(\bar{E}[\ell'])$. Hence (\ref{eq:upfin0}) leads to an exact sequence
\begin{equation}
\label{eq:upfin1}
0 \to \frac{\bar{E}[\ell']}{(\Phi^{\ell'}-1)(\bar{E}[\ell'])} \to 
\frac{\Gamma/\ell'\,T_\ell(\bar{E})}{(\Phi^{\ell'}-1)(\bar{E}[\ell'])\cdot\widehat{\langle \Phi^{\ell'}\rangle}} \to \frac{\widehat{\langle \Phi\rangle}}{\widehat{\langle \Phi^{\ell'}\rangle}} \to 1.
\end{equation}
We define
\begin{eqnarray}
\label{eq:needthis2}
&\overline{\mathcal{N}_\ell}:=(\Phi^{\ell'}-1)(\bar{E}[\ell']), \quad 
\overline{\mathcal{T}_\ell}:=\displaystyle \frac{\bar{E}[\ell']}{\overline{\mathcal{N}_\ell}} , \\[1ex]
\label{eq:needthis3}
&\displaystyle\langle \,\overline{\Phi}_\ell\, \rangle:=\frac{\widehat{\langle \Phi\rangle}}{\widehat{\langle \Phi^{\ell'}\rangle}}  ,
\quad\mbox{and}\quad
\overline{G}_\ell:=\frac{\Gamma/\ell'\,T_\ell(\bar{E})}{\overline{\mathcal{N}_\ell} \cdot\widehat{\langle \Phi^{\ell'}\rangle} }.&
\end{eqnarray}
Letting $\xi_\ell:\langle \overline{\Phi}_\ell\rangle\to \mathrm{Aut}(\overline{\mathcal{T}_\ell})$ be the group homomorphism induced by (\ref{eq:upfin1}), $\overline{G}_\ell$ is the semidirect product 
\begin{equation}
\label{eq:pi13}
\overline{G}_\ell = \overline{\mathcal{T}_\ell}\rtimes_{\xi_\ell} \langle \overline{\Phi}_\ell\rangle.
\end{equation}
We view $\overline{\mathcal{T}_\ell}$ as a subgroup of $\overline{G}_\ell$ and we view $\overline{\Phi}_\ell$ as an element of $\overline{G}_\ell$ of order $\ell'$. Note that the commutator subgroup ${\overline{G}_\ell}'$ of $\overline{G}_\ell$ is contained in the abelian subgroup $\overline{\mathcal{T}_\ell}$ which implies that ${\overline{G}_\ell}''$ is trivial. 

For all $\ell\ge 3$, let $H_1$ be the subgroup of $U_4(\mathbb{Z}/\ell)$ defined in Remark \ref{rem:U4!} and let 
$$(p_1,p_2,p_3):\quad U_4(\mathbb{Z}/\ell)\to (\mathbb{Z}/\ell)^{\oplus 3}$$
be the homomorphism that sends each matrix $M=M(a_1,a_2,a_3,u,v,w)$ in (\ref{eq:matrixnotation}) to the triple $(a_1,a_2,a_3)$.  If $\chi_1,\chi_2,\chi_3:\pi_1(E)\to \mathbb{Z}/\ell$ are non-trivial characters, then they factor through the maximal elementary abelian $\ell$-quotient of $\pi_1(E)$, and hence through $\overline{G}_\ell$. Recall that the group $U_4(\mathbb{Z}/\ell)$ has exponent $9$ if $\ell=3$, and it has exponent $\ell$ if $\ell>3$. Hence we see, similarly to the discussion following (\ref{eq:rhobar}), that the triple Massey product $\langle \chi_1,\chi_2,\chi_3\rangle$ contains zero if and only if $\chi_1\cup\chi_2=\chi_2\cup\chi_3=0$ and the map $(\chi_1,\chi_2,\chi_3):\overline{G}_\ell\to (\mathbb{Z}/\ell)^{\oplus 3}$ can be lifted to a continuous group homomorphism $\rho: \overline{G}_\ell \to U_4(\mathbb{Z}/\ell)$ such that $(\chi_1,\chi_2,\chi_3) =  (p_1,p_2,p_3) \circ \rho$. 

\subsection{Suppose the $\ell$-torsion $\bar{E}[\ell]$ is defined over $\mathbb{F}_q$.}
\label{ss:easy1} 

By (\ref{eq:ohyeah!}) and Lemma \ref{lem:nice!}, we are reduced to the case when $\ell=3$, $\ell'=9$, and $\chi_1=\chi_2=\chi_3$ is given by a single character $\chi: \overline{G}_3 \to\mathbb{Z}/3$. Since $\Phi-1$ acts trivially on $\bar{E}[3]$, $(\Phi-1)^2$ acts trivially on $\bar{E}[9]$, which implies as in \S\ref{s:triple} that $\overline{\mathcal{T}_3}=\bar{E}[9]$. Since there is precisely one cubic extension of $\mathbb{F}_q$ inside $\overline{\mathbb{F}_q}$, we get the following simplification of Theorem \ref{thm:classifiynontrivial!}.

\begin{theorem}
\label{thm:ellfin3case1}
Suppose $E$ is an elliptic curve over a finite field $\mathbb{F}_q$ such that $q$ is not divisible by $3$ and such that the $3$-torsion $\bar{E}[3]$ is defined over $\mathbb{F}_q$. Let $\chi:\overline{G}_3\to\mathbb{Z}/3$ be a character. Then $\langle \chi,\chi,\chi \rangle$ does not contain zero if and only if the restriction of $\chi$ to $\bar{E}[9]$ is non-zero and there exists an element $a\in\bar{E}[9]-3\bar{E}[9]$ such that $\chi(a)=0$ and $\Phi_3(a)\not\in (\mathbb{Z}/9)\, a$.
\end{theorem}

For examples of the situation discussed in Theorem \ref{thm:ellfin3case1}, see Remark \ref{rem:Cebotarev}.

\subsection{Suppose the $\ell$-torsion $\bar{E}[\ell]$ is not defined over $\mathbb{F}_q$.}
\label{ss:hard2} 

This means that the set of fixed points in $\bar{E}[\ell]$ under the action of $\Phi$ has either order $1$ or $\ell$. If this set has order $1$, then it follows that the maximal elementary abelian $\ell$-quotient group of $\pi_1(E)$ is a group of order $\ell$ given by $\widehat{\langle \Phi\rangle}/\widehat{\langle \Phi^\ell\rangle}$. Hence all characters in $\HH^1(E,\mathbb{Z}/\ell)$ are in $\HH^1(\mathbb{F}_q,\mathbb{Z}/\ell)$. Since $\HH^2(\mathbb{F}_q,\mathbb{Z}/\ell)=0$, every triple Massey product that is non-empty contains zero.

For the remainder of this subsection, we assume that the set of fixed points in $\bar{E}[\ell]$ under the action of $\Phi$ has order $\ell$. We need the following remark.

\begin{remark}
\label{rem:ohyes!}
Let $\ell \ge 3$ and let $\Gamma$ be as in (\ref{eq:Gamma}). Letting $\lambda:\widehat{\langle \Phi\rangle}\to \mathrm{Aut}(T_\ell(\bar{E})/\ell \,T_\ell(\bar{E})) = \mathrm{Aut}(\bar{E}[\ell])$ be the group homomorphism induced by the sequence (\ref{eq:upfinfirst}), $\Gamma/\ell \,T_\ell(\bar{E})$ is the semidirect product
\begin{equation}
\label{eq:Gammaquotient}
\Gamma/\ell\, T_\ell(\bar{E}) = \bar{E}[\ell]\rtimes_\lambda \widehat{\langle \Phi \rangle}.
\end{equation}
Since we assume that the set of fixed points in $\bar{E}[\ell]$ under the action of $\Phi$ has order $\ell$, there exists a basis $\{\overline{m}_1,\overline{m}_2\}$ of $\bar{E}[\ell]$ over $\mathbb{Z}/\ell$ such that the action of $\Phi$ on $\bar{E}[\ell]$ with respect to this basis is given by the matrix $\overline{A}_\Phi \in\mathrm{Aut}(\bar{E}[\ell])=\mathrm{GL}_2(\mathbb{Z}/\ell)$, where 
\begin{eqnarray}
\label{eq:TWO1}
\mbox{either} \quad \overline{A}_\Phi &=& \begin{pmatrix} 1&0\\0&\varepsilon\end{pmatrix}\quad\mbox{for some element $\varepsilon\in(\mathbb{Z}/\ell)^\times-\{1\}$}, \\
\label{eq:TWO2}
\mbox{or} \quad \overline{A}_\Phi &=& \begin{pmatrix} 1&1\\0&1\end{pmatrix}.
\end{eqnarray}
In both cases (\ref{eq:TWO1}) and (\ref{eq:TWO2}), the subgroup of $\bar{E}[\ell]$ generated by $\overline{m}_1$ equals the set of fixed points in $\bar{E}[\ell]$ under the action of $\Phi$. In the case (\ref{eq:TWO1}), the image of $(\Phi-1)$ on $\bar{E}[\ell]$ is given by $(\mathbb{Z}/\ell)\overline{m}_2$, whereas in the case (\ref{eq:TWO2}), the image of $(\Phi-1)$ on $\bar{E}[\ell]$ is given by $(\mathbb{Z}/\ell)\overline{m}_1$. Therefore, every character $\chi:\Gamma/\ell \,T_\ell(\bar{E})\to\mathbb{Z}/\ell$ satisfies $\chi(\overline{m}_2)=0$ if $\overline{A}_\Phi$ is as in (\ref{eq:TWO1}) and it satisfies $\chi(\overline{m}_1)=0$ if $\overline{A}_\Phi$ is as in (\ref{eq:TWO2}).
\end{remark}

The following result pins down the structure of $\overline{G}_\ell$ for $\ell\ge 3$.

\begin{lemma}
\label{lem:Gell}
Let $\ell\ge 3$, let $\ell'$ be as in $(\ref{eq:ell'})$, and let  $\overline{G}_\ell$ be as in $(\ref{eq:pi13})$. Extend the action of $\Phi$ on $\bar{E}[\ell]$ from $(\ref{eq:Gammaquotient})$ to an action of the integral group ring $\mathbb{Z}[\Phi]$ on $\bar{E}[\ell]$.
\begin{itemize}
\item[(a)] If $(\Phi-1)^2$ does not act as zero on $\bar{E}[\ell]$, i.e. $\overline{A}_\Phi$ is given as in $(\ref{eq:TWO1})$ with respect to some basis of $\bar{E}[\ell]$ over $\mathbb{Z}/\ell$, then $\overline{\mathcal{T}_\ell}\cong \mathbb{Z}/\ell'$ and 
$$\overline{G}_\ell= (\mathbb{Z}/\ell')\rtimes_{\xi_\ell }\langle\overline{\Phi}_\ell\rangle$$
where $\xi_\ell:\langle\overline{\Phi}_\ell\rangle\to (\mathbb{Z}/\ell')^\times$ is given by $\xi_\ell(\overline{\Phi}_\ell)=1+\ell\alpha $
for a certain $\alpha\in \mathbb{Z}/\ell'$. If $\ell=3$ then there is a unique $\alpha\in\{0,1,2\}$ such that $\xi_3(\overline{\Phi}_3)=1+3\alpha$, and if $\ell>3$ then $\xi_\ell(\overline{\Phi}_\ell)= 1$ and we let $\alpha=0$.

\item[(b)] If $(\Phi-1)^2$ acts as zero on $\bar{E}[\ell]$, i.e. $\overline{A}_\Phi$ is given as in $(\ref{eq:TWO2})$ with respect to some basis of $\bar{E}[\ell]$ over $\mathbb{Z}/\ell$, then $\overline{\mathcal{T}_\ell}\cong \bar{E}[\ell']$ and
$$\overline{G}_\ell= \bar{E}[\ell']\rtimes_{\xi_\ell }\langle\overline{\Phi}_\ell\rangle$$
where $\xi_\ell:\langle\overline{\Phi}_\ell\rangle\to \mathrm{GL}_2(\mathbb{Z}/\ell')$ is given by
$$\xi_\ell(\overline{\Phi}_\ell)= \begin{pmatrix} 1+\ell\alpha & 1+\ell\beta\\\ell\gamma&1+\ell\delta\end{pmatrix}$$
for certain $\alpha,\beta,\gamma,\delta\in \mathbb{Z}/\ell'$. If $\ell=3$ then there are unique such $\alpha,\beta,\gamma,\delta$ in $\{0,1,2\}$. On the other hand, if $\ell>3$ then $\xi_\ell(\overline{\Phi}_\ell)= \begin{pmatrix} 1&1\\0&1\end{pmatrix}$ and we let $\alpha=\beta=\gamma=\delta=0$.
\end{itemize}
\end{lemma}

\begin{proof}
Suppose first that we are in part (a), i.e. there exists a basis $\{\overline{m}_1,\overline{m}_2\}$ of $\bar{E}[\ell]$ over $\mathbb{Z}/\ell$ such that the action of $\Phi$ on $\bar{E}[\ell]$ with respect to this basis is given by the matrix $\overline{A}_\Phi$ in (\ref{eq:TWO1}).

If $\ell=3$ then $\ell'=9$ and $\varepsilon=2$. In this case, let $\{m_1,m_2\}$ be a basis of $\bar{E}[9]$ that reduces to the basis $\{\overline{m}_1,\overline{m}_2\}$ modulo 3. Then there exist $\alpha,\beta,\gamma,\delta\in \mathbb{Z}/9$ such that the action of $\Phi$ on $\bar{E}[9]$ is given by the matrix
$$A_\Phi=\begin{pmatrix} 1+3\alpha & 3\beta\\3\gamma&2+3\delta\end{pmatrix}.$$ 
Hence $A_\Phi^9-I=\begin{pmatrix} 0 & 3\beta\\3\gamma&7\end{pmatrix}$. In particular, we have $\overline{\mathcal{N}_3}=(\mathbb{Z}/9) (3\beta\,m_1+7m_2)=(\mathbb{Z}/9) (3\beta\,m_1+m_2)$ in (\ref{eq:needthis2}), and hence $\overline{\mathcal{T}_3}\cong \mathbb{Z}/9$. Moreover, since $\Phi(m_1)\equiv(1+3\alpha)\,m_1 \mod \overline{\mathcal{N}_3}$, we obtain that $\xi_3(\overline{\Phi}_3)=1+3\alpha$.

If $\ell>3$ then $\ell'=\ell$. In this case, the action of $\Phi$ on $\bar{E}[\ell]$ is given by the matrix $\overline{A}_\Phi$ in (\ref{eq:TWO1}). Hence $\overline{A}_\Phi^\ell-I=\begin{pmatrix} 0 & 0\\ 0&\varepsilon-1\end{pmatrix}$. Since $\varepsilon-1\in(\mathbb{Z}/\ell)^\times$, we have $\overline{\mathcal{N}_\ell}=(\mathbb{Z}/\ell) \overline{m}_2$ in (\ref{eq:needthis2}), and hence $\overline{\mathcal{T}_\ell}\cong \mathbb{Z}/\ell$. Moreover, since $\Phi(m_1)=m_1$, we obtain that $\xi_\ell(\overline{\Phi}_\ell)=1$. This completes the proof of part (a).

Suppose next that we are in part (b), i.e. there exists a basis $\{\overline{m}_1,\overline{m}_2\}$ of $\bar{E}[\ell]$ over $\mathbb{Z}/\ell$ such that the action of $\Phi$ on $\bar{E}[\ell]$ with respect to this basis is given by the matrix $\overline{A}_\Phi$ in (\ref{eq:TWO2}). 

If $\ell=3$ then $\ell'=9$. In this case, let $\{m_1,m_2\}$ be a basis of $\bar{E}[9]$ that reduces to the basis $\{\overline{m}_1,\overline{m}_2\}$ modulo 3. Then there exist $\alpha,\beta,\gamma,\delta\in \mathbb{Z}/9$ such that the action of $\Phi$ on $\bar{E}[9]$ is given by the matrix
$$A_\Phi=\begin{pmatrix} 1+3\alpha & 1+3\beta\\3\gamma&1+3\delta\end{pmatrix}.$$ 
Hence $A_\Phi^9-I$ is the zero matrix. It follows that $\overline{\mathcal{N}_3}=0$ in (\ref{eq:needthis2}), and hence $\overline{\mathcal{T}_3}=\bar{E}[9]$. In particular, $\xi_3(\overline{\Phi}_3)$ has the desired shape.

If $\ell>3$ then $\ell'=\ell$. In this case, the action of $\Phi$ on $\bar{E}[\ell]$ is given by the matrix $\overline{A}_\Phi$ in (\ref{eq:TWO2}). Hence $\overline{A}_\Phi^\ell-I$ is the zero matrix. It follows that $\overline{\mathcal{N}_\ell}=0$ in (\ref{eq:needthis2}), and hence $\overline{\mathcal{T}_\ell}=\bar{E}[\ell]$. In particular, $\xi_\ell(\overline{\Phi}_\ell)$ has the desired shape. This completes the proof of part (b).
\end{proof}

We have the following result on cup products:

\begin{lemma}
\label{lem:cuppdegenerate}
Let $\ell\ge 3$. Suppose $E$ is an elliptic curve over a finite field $\mathbb{F}_q$ such that $q$ is not divisible by $\ell$ and such that the set of fixed points in $\bar{E}[\ell]$ under the action of $\Phi$ has order $\ell$. 
Let $\chi_1,\chi_2\in \HH^1(E,\mathbb{Z}/\ell)=\mathrm{Hom}(\pi_1(E),\mathbb{Z}/\ell)$ be non-trivial characters. 
\begin{itemize}
\item[(a)] If $(\Phi-1)^2$ does not act as zero on $\bar{E}[\ell]$, then $\chi_1\cup\chi_2=0$ if and only if there exists $a\in (\mathbb{Z}/\ell)^\times$ such that $\chi_2=a\chi_1$.
\item[(b)] If $(\Phi-1)^2$ acts as zero on $\bar{E}[\ell]$, then we always have $\chi_1\cup\chi_2=0$. 
\end{itemize}
\end{lemma}

\begin{proof}
Since $\ell\ge 3$, every non-identity element of $U_3(\mathbb{Z}/\ell)$ has order $\ell$. It follows that $\chi_1 \cup \chi_2 = 0$  as elements of $\HH^1(\pi_1(E),\mathbb{Z}/\ell)$ if and only if this is so when we consider them as elements of $\HH^1(\overline{G}_\ell/(\overline{G}_\ell)^\ell, \mathbb{Z}/\ell)$ with cup product in $\HH^2(\overline{G}_\ell/(\overline{G}_\ell)^\ell, \mathbb{Z}/\ell)$.

Suppose first that we are in part (a), i.e. 
$$\overline{G}_\ell/(\overline{G}_\ell)^\ell \cong \mathbb{Z}/\ell \times \mathbb{Z}/\ell\quad\mbox{for all $\ell\ge 3$}.$$
In particular, we write $\overline{G}_\ell/(\overline{G}_\ell)^\ell$ additively. The cup product on $\HH^1(\mathbb{Z}/\ell \times \mathbb{Z}/\ell, \mathbb{Z}/\ell)$ is a non-degenerate alternating bilinear form on a two-dimensional vector space over $\mathbb{Z}/\ell$ with values in 
$\HH^2(\mathbb{Z}/\ell \times \mathbb{Z}/\ell, \mathbb{Z}/\ell)$.  So it factors through the determinant and vanishes exactly on pairs that span the same space. This proves part (a).

Suppose next that we are in part (b), i.e. 
$$\overline{G}_\ell/(\overline{G}_\ell)^\ell \cong \bar{E}[\ell] \rtimes_{\overline{\xi}} \langle \overline{\overline{\Phi}}\rangle \quad\mbox{for all $\ell\ge 3$}$$
where $\langle\overline{\overline{\Phi}}\rangle = \widehat{\langle \Phi \rangle} / \widehat{\langle \Phi^\ell \rangle}$ and there exists a basis $\{\overline{m}_1,\overline{m}_2\}$ of $\bar{E}[\ell]$ such that, with respect to this basis, $\overline{\xi}(\overline{\overline{\Phi}})=\overline{A}_\Phi$ as in (\ref{eq:TWO2}). We want to show that $\chi_1\cup\chi_2=0$, which is equivalent to the statement that there exists a map $\kappa: \overline{G}_\ell/(\overline{G}_\ell)^\ell\to \mathbb{Z}/\ell$ such that the map 
\begin{equation}
\label{eq:rhoagain}
\begin{array}{cccc}
\rho:& \overline{G}_\ell/(\overline{G}_\ell)^\ell &\to& U_3(\mathbb{Z}/\ell)\\
&\overline{g} & \mapsto & \left(\begin{array}{ccc}
1&\chi_1(\overline{g})&\kappa(\overline{g})\\
0&1&\chi_2(\overline{g})\\
0&0&1\end{array}\right) 
\end{array}
\end{equation}
is a group homomorphism.  Since the cup product is alternating and $\HH^1(\overline{G}_\ell/(\overline{G}_\ell)^\ell,\mathbb{Z}/\ell)$ has dimension two, it will suffice to consider the case in which $\{\chi_1,\chi_2\}$ is the dual basis over $\mathbb{Z}/\ell$ to the basis for the maximal abelian quotient of $\overline{G}_\ell/(\overline{G}_\ell)^\ell$ formed by the images of $\overline{m}_2$ and $\overline{\overline \Phi}$.   One then checks by a commutator computation that $\kappa$ can be defined by 
$$\kappa(\overline{m}_1)=\chi_1(\overline{\overline{\Phi}})\chi_2(\overline{m}_2) - \chi_2(\overline{\overline{\Phi}})\chi_1(\overline{m}_2)\quad\mbox{and}\quad
\kappa(\overline{m}_2)=0=\kappa(\overline{\overline{\Phi}}).$$
Note that in this case $\rho$ is a group isomorphism, which completes the proof.
\end{proof}

As a consequence, we obtain the following result when $\overline{G}_\ell$ is as in part (a) of Lemma \ref{lem:Gell}:

\begin{lemma}
\label{lem:Gelleasy}
Suppose $\ell\ge 3$ and that $E$ is an elliptic curve over a finite field $\mathbb{F}_q$ such that $q$ is not divisible by $\ell$ and such that the set of fixed points in $\bar{E}[\ell]$ under the action of $\Phi$ has order $\ell$. Moreover, suppose that $(\Phi-1)^2$ does not act as zero on $\bar{E}[\ell]$. Let $\chi_1,\chi_2,\chi_3\in\HH^1(E,\mathbb{Z}/\ell)$ be characters such that $\chi_1\cup\chi_2=\chi_2\cup\chi_3=0$. Then $\langle \chi_1,\chi_2,\chi_3 \rangle$ contains zero.
\end{lemma}

\begin{proof}
If any of $\chi_1,\chi_2,\chi_3$ is trivial, then $\langle \chi_1,\chi_2,\chi_3\rangle$ contains zero. Suppose now that none of these characters is trivial. Since $\chi_1\cup\chi_2=\chi_2\cup\chi_3=0$, we are, by (\ref{eq:ohyeah!}) and Lemma \ref{lem:cuppdegenerate}, reduced to consider the case when $\chi_1=\chi_2=\chi_3$ is a single character $\chi$. If the restriction $\overline{\chi}$ to $\HH^1(\bar{E},\mathbb{Z}/\ell)$ is trivial, then $\langle \chi,\chi,\chi\rangle$ contains zero since $\HH^2(\mathbb{F}_q,\mathbb{Z}/\ell)=0$.

Suppose now that $\overline{\chi}$ is non-trivial. Using the properties of $U_4(\mathbb{Z}/\ell)$, we can replace $\pi_1(E)$ by $\overline{G}_\ell$ in our arguments and assume that $\chi:\overline{G}_\ell\to\mathbb{Z}/\ell$. As in part (a) of Lemma \ref{lem:Gell}, we write
$$\overline{G}_\ell = (\mathbb{Z}/\ell')\rtimes_{\xi_\ell}\langle\overline{\Phi}_\ell\rangle$$
where $\xi_\ell(\overline{\Phi}_\ell)=1+\ell\alpha$ for some $\alpha\in \mathbb{Z}/\ell'$. Moreover, if $\ell=3$ then $\ell'=9$ and we choose $\alpha\in\{0,1,2\}$, and if $\ell>3$ then $\ell'=\ell$ and we choose $\alpha=0$. 

Let $m$ be a generator of $\mathbb{Z}/\ell'$. In particular, since $\overline{\chi}\ne 0$, we have that $\chi(m)\neq 0$. By replacing $m$ by a multiple if necessary, we can assume without loss of generality that $\chi(m)=1$. Let $\varphi\in \{0,1,\ldots,\ell-1\}$ be such that $\chi(\overline{\Phi}_\ell)\equiv\varphi\mod \ell$. We define a map $\rho:\overline{G}_\ell\to U_4(\mathbb{Z}/\ell)$ by
$$\rho(m)=\left(\begin{array}{cccc} 1&1&0&0\\0&1&1&0\\0&0&1&1\\0&0&0&1\end{array}\right) \quad\mbox{and}\quad
\rho(\overline{\Phi}_\ell)=\left(\begin{array}{cccc} 1&0&\alpha&0\\0&1&0&0\\0&0&1&0\\0&0&0&1\end{array}\right)\,\rho(m)^\varphi.$$
Then it follows that
$$[\rho(\overline{\Phi}_\ell),\rho(m)]=
[\rho(\overline{\Phi}_\ell)\rho(m)^{-\varphi},\rho(m)]=
\left(\begin{array}{cccc} 1&0&0&\alpha\\0&1&0&0\\0&0&1&0\\0&0&0&1\end{array}\right) = \rho(m)^{\ell\alpha}=
\rho((\overline{\Phi}_\ell-1)\,m)$$
where the second equality follows from (\ref{eq:commutator}), the third equality follows from (\ref{eq:powers}) and our choice of $\alpha$, and the last equality follows since $(\overline{\Phi}_\ell-1)\,m = \ell\alpha\,m$. This shows that $\rho$ is a group homomorphism, completing the proof.
\end{proof}

\begin{example}
\label{Ex:burp1}
For an example of the situation discussed in Lemma \ref{lem:Gelleasy}, let $\ell\ge 3$ be a prime number such that there exists an elliptic curve $\E$ over $\mathbb{Q}$ that has an $\ell$-torsion point. By \cite[Thm. 2]{Mazur1978}, $\ell\in\{3, 5, 7\}$. Let $p$ be a rational prime such that $p\equiv 2\!\mod \ell$ and such that $\E$ has good reduction modulo $p$ (using a Cebotarev argument, there are infinitely such $p$). This results in an elliptic curve $E$ over $\mathbb{F}_p$. By a classical result by Hasse, the action of $\Phi$ on $\bar{E}[\ell]$ has determinant $p \!\mod \ell \equiv 2 \!\mod \ell$. Hence one eigenvalue of this action is $1 \!\mod \ell$ and the other is $2 \!\mod \ell$, which means the conditions of Lemma \ref{lem:Gelleasy} are satisfied. 

It is easy to generalize to larger prime numbers $\ell$ by considering elliptic curves $\E$ over $\mathbb{Q}$ with good reduction modulo $p$ for which the image of the Galois representation on $\bar{\E}[\ell]$ is as large as possible (see \cite{SerreLAdic}), and by then passing to a suitable finite extension of $\mathbb{Q}$ containing an $\ell$-torsion point of $\bar{\E}$.
\end{example}

When $\overline{G}_\ell$ is as in part (b) of Lemma \ref{lem:Gell}, we obtain the following result:

\begin{proposition}
\label{prop:ellfin3case2}
Suppose $\ell\ge 3$ and that $E$ is an elliptic curve over a finite field $\mathbb{F}_q$ such that $q$ is not divisible by $\ell$ and such that the set of fixed points in $\bar{E}[\ell]$ under the action of $\Phi$ has order $\ell$. Moreover, suppose that $(\Phi-1)^2$ acts as zero on $\bar{E}[\ell]$. Let $\chi_1,\chi_2,\chi_3\in\HH^1(E,\mathbb{Z}/\ell)$ be non-trivial characters.

Then $\chi_1\cup\chi_2=\chi_2\cup\chi_3=0$. Moreover, $\langle \chi_1,\chi_2,\chi_3 \rangle$ contains zero if and only if there exists $m\in\bar{E}[\ell']$ whose image $\overline{m}\in\bar{E}[\ell]$ is not fixed by $\Phi$ such that the following two conditions hold when we write $x_i=\chi_i(m)$ and $\varphi_i=\chi_i(\overline{\Phi}_\ell)$ for $1\le i\le3$: 
\begin{itemize}
\item[(1)] $(\varphi_2x_3-\varphi_3x_2)\,x_1 - (\varphi_1x_2-\varphi_2x_1)\,x_3 \equiv 0\mod \ell$, and
\item[(2)] $(\varphi_2x_3-\varphi_3x_2)\,\varphi_1 - (\varphi_1x_2-\varphi_2x_1)\,\varphi_3 \equiv c \,x_1x_2x_3\mod \ell$, \\
where $c \in \mathbb{Z}$, $c=0$ if $\ell>3$, and $(\Phi-1)^2(m)\equiv 3\,c \,m\mod \langle \,3(\Phi-1)(m)\,\rangle$ if $\ell = 3$. 
\end{itemize}
\end{proposition}

\begin{proof}
The first statement follows from part (b) of Lemma \ref{lem:cuppdegenerate}.
Let $m\in\bar{E}[\ell']$ be such that its image $\overline{m}\in\bar{E}[\ell]$ is not fixed by $\Phi$. Define $m':=(\Phi-1)(m)$, so the image $\overline{m'}\in\bar{E}[\ell]$ is not zero. Since the action of $(\Phi-1)^2$ is zero on $\bar{E}[\ell]$, it follows that $\overline{m'}$ is fixed by $\Phi$. In particular, $\chi_i(m')=0$ for all $1\le i \le 3$ by Remark \ref{rem:ohyes!}, implying that $\{m',m\}$ is a basis of $\bar{E}[\ell']$ over $\mathbb{Z}/\ell'$. It follows, as in part (b) of Lemma \ref{lem:Gell}, that the action of $\Phi$ on $\bar{E}[\ell']$ with respect to this basis $\{m',m\}$ is given by a matrix of the form
\begin{equation}
\label{eq:matrix!}
\xi_\ell(\overline{\Phi}_\ell) = \begin{pmatrix} 1+\ell\alpha & 1+\ell\beta\\\ell\gamma&1+\ell\delta\end{pmatrix}
\end{equation}
for certain $\alpha,\beta,\gamma,\delta\in\mathbb{Z}/\ell'$. Moreover, if $\ell=3$ then $\ell'=9$ and we choose $\alpha,\beta,\gamma,\delta\in\{0,1,2\}$, and if $\ell>3$ then $\ell'=\ell$ and we choose $\alpha=\beta=\gamma=\delta=0$. Hence, we obtain
\begin{equation}
\label{eq:independent}
(\Phi-1)^2(m)=(\Phi-1)(m') = \left\{\begin{array}{cl}
3\alpha\,m'+3\gamma\,m&\mbox{if $\ell=3$},\\
0&\mbox{if $\ell>3$}.
\end{array}\right.
\end{equation}
Define
\begin{equation}
\label{eq:c}
c=\left\{\begin{array}{cl}
\gamma&\mbox{if $\ell=3$},\\
0&\mbox{if $\ell>3$}.
\end{array}\right.
\end{equation}
In particular, $(\Phi-1)^2(m)\equiv 3\,c \,m\mod \langle \,3(\Phi-1)(m)\,\rangle$ if $\ell = 3$. 

The Massey product $\langle \chi_1,\chi_2,\chi_3\rangle$ contains zero if and only if there exists a group homomorphism $\rho:\overline{G}_\ell\to U_4(\mathbb{Z}/\ell)$ such that
$$\rho(m')=\left(\begin{array}{cccc} 1&0&r&s\\0&1&0&t\\0&0&1&0\\0&0&0&1\end{array}\right), \quad
\rho(m)=\left(\begin{array}{cccc} 1&x_1&u&v\\0&1&x_2&w\\0&0&1&x_3\\0&0&0&1\end{array}\right) \quad\mbox{and}\quad
\rho(\overline{\Phi})=\left(\begin{array}{cccc} 1&\varphi_1&d&e\\0&1&\varphi_2&f\\0&0&1&\varphi_3\\0&0&0&1\end{array}\right)$$
for certain $r,s,t,u,v,w,d,e,f\in\mathbb{Z}/\ell$. 
Using (\ref{eq:matrix!}), such a $\rho$ exists if and only if the following three relations are satisfied in $U_4(\mathbb{Z}/\ell)$:
\begin{eqnarray}
\label{eq:1in3}
{[\rho(m'),\rho(m)]} &=& I_4,\\
\label{eq:2in3}
{[\rho(\overline{\Phi}),\rho(m')]}  &=& \rho(m')^{\ell\alpha}\rho(m)^{\ell\gamma}\;=\;\rho(m)^{\ell\gamma},\\
\label{eq:3in3}
{[\rho(\overline{\Phi}),\rho(m)]}  &=& \rho(m')^{1+\ell\beta}\rho(m)^{\ell\delta}\;=\;\rho(m')\rho(m)^{\ell\delta},
\end{eqnarray}
where $I_4$ is the identity matrix in $U_4(\mathbb{Z}/\ell)$ and the second equality in both (\ref{eq:2in3}) and (\ref{eq:3in3}) follows since $\rho(m')$ has order dividing $\ell$ in $U_4(\mathbb{Z}/\ell)$ for all $\ell\ge 3$. Using  (\ref{eq:powers}) and (\ref{eq:commutator}), we see that equations (\ref{eq:1in3}), (\ref{eq:2in3}) and (\ref{eq:3in3}) are equivalent to the following equalities in $\mathbb{Z}/\ell$:
\begin{eqnarray}
\label{eq:eq1*}
rx_3&=&tx_1,\\
\label{eq:eq2*}
\gamma \,x_1x_2x_3 &=&t\varphi_1-r\varphi_3,\\
\label{eq:eq3*}
r&=&\varphi_1x_2-\varphi_2x_1,\\
\label{eq:eq4*}
t&=&\varphi_2x_3-\varphi_3x_2,\\
\label{eq:eq5*}
s+\delta\,x_1x_2x_3 &=&(\varphi_1w-fx_1)-(\varphi_3u-dx_3)-(\varphi_1x_2-\varphi_2x_1)(\varphi_3+x_3).
\end{eqnarray}
It follows that $\langle \chi_1,\chi_2,\chi_3\rangle$ contains zero if and only if there exists at least one choice of $r,s,t,u,v,w,d,e,f\in\mathbb{Z}/\ell$ such that all equations (\ref{eq:eq1*}) - (\ref{eq:eq5*}) are satisfied. 
Letting $u=w=d=f=0$ and $s=-\delta\,x_1x_2x_3-(\varphi_1x_2-\varphi_2x_1)(\varphi_3+x_3)$ satisfies (\ref{eq:eq5*}).
Since (\ref{eq:eq1*}) - (\ref{eq:eq4*}) only involve $r$ and $t$, we see that $\langle \chi_1,\chi_2,\chi_3\rangle$ contains zero if and only if there exist $r,t\in\mathbb{Z}/\ell$ such that all equations (\ref{eq:eq1*}) - (\ref{eq:eq4*}) are satisfied. Since, by (\ref{eq:c}), $c\equiv \gamma \mod 3$ if $\ell=3$ and since $c=\gamma=0$ if $\ell>3$, substituting (\ref{eq:eq3*}) and (\ref{eq:eq4*}) into (\ref{eq:eq1*}) and (\ref{eq:eq2*}) finishes the proof of Proposition \ref{prop:ellfin3case2}.
\end{proof}

We now prove Theorem~\ref{cor:forallell} as a consequence of Proposition~\ref{prop:ellfin3case2}.

\begin{proof}[Proof of Theorem $\ref{cor:forallell}$]
Let $E_t$ be the elliptic curve over $\Q(t)$ defined by \eqref{eq:igusacurve}. As stated in Example~\ref{ex:igusa}, Igusa has proved that, for all primes $\ell$, the Galois representation 
$$\Gal(\Q(\bar{E}_t[\ell])/\Q(t)) \to \GL_2(\Z/\ell)$$
is surjective, hence bijective. According to Hilbert's irreducibility theorem, these Galois groups remain the same for infinitely many rational specializations of the parameter $t$. Therefore, the prime $\ell$ being fixed, one obtains infinitely many (non-isomorphic) elliptic curves $\E$ over $\Q$ such that 
$$\Gal(\Q(\bar{\E}[\ell])/\Q) \simeq \GL_2(\Z/\ell).$$

Using a Cebotarev argument, it follows that there are infinitely many rational primes $p\ne \ell$ such that $\E$ has good reduction modulo $p$ and such that if $\mathfrak{p}$ is a prime above $p$ in $\mathbb{Q}(\bar{\E}[\ell])$ then $\mathfrak{p}$ is unramified over $p$ and the Frobenius action associated to $\mathfrak{p}$ is in the conjugacy class of the matrix $\begin{pmatrix} 1&1\\0&1\end{pmatrix}$ in $\mathrm{GL}_2(\mathbb{Z}/\ell)$.

Let $\ell>3$ and $p$ be prime numbers as above, and let $E$ be the reduction of $\E$ modulo $p$. Then $E$ is an elliptic curve over $\mathbb{F}_p$ such that the set of fixed points in $\bar{E}[\ell]$ under the action of $\Phi$ has order $\ell$ and such that $(\Phi-1)^2$ acts as zero on $\bar{E}[\ell]$. As in Lemma \ref{lem:Gell}, there exists a basis $\{m',m\}$ of $\bar{E}[\ell]$ over $\mathbb{Z}/\ell$ such that the action of $\Phi$ on $\bar{E}[\ell]$ with respect to this basis is given by $\xi(\overline{\Phi}_\ell) = \begin{pmatrix}1&1\\0&1\end{pmatrix}$. In particular, $m\in\bar{E}[\ell]$ is not fixed by $\Phi$. We define a group homomorphism $\chi:\overline{G}_\ell\to \mathbb{Z}/\ell$ by 
\begin{equation}
\label{eq:chiex}
\chi(m')=\chi(\overline{\Phi}_\ell)=0\quad \mbox{and} \quad \chi(m)=1.
\end{equation} 
Note that this gives indeed a group homomorphism since $\overline{G}_\ell/\langle m'\rangle \cong \mathbb{Z}/\ell\times \mathbb{Z}/\ell$.

Let $\chi_1,\chi_3\in \HH^1(E,\mathbb{Z}/\ell)$ be non-trivial characters, given by group homomorphisms $\overline{G}_\ell\to \mathbb{Z}/\ell$ whose kernels are equal to $\bar{E}[\ell]$. In particular, $x_1=\chi_1(m)=0$ and $x_3=\chi_3(m)=0$. Since $\chi_1$ and $\chi_3$ are non-trivial, we have that $\varphi_1=\chi_1(\overline{\Phi}_\ell)\ne 0$ and $\varphi_3=\chi_3(\overline{\Phi}_\ell)\ne 0$. Let $\chi_2:=\chi$ be as in (\ref{eq:chiex}); in particular,  $x_2=\chi_2(m)=1$. Then condition (1) from Proposition \ref{prop:ellfin3case2} is satisfied, whereas condition (2) is not satisfied since
$$(\varphi_2x_3-\varphi_3x_2)\,\varphi_1 - (\varphi_1x_2-\varphi_2x_1)\,\varphi_3 \equiv -2\,\varphi_1\varphi_3\not\equiv 0\mod \ell.$$
In other words, $\langle \chi_1,\chi_2,\chi_3\rangle$ is not empty and does not contain zero.
\end{proof}

\begin{example}
Considering the same construction as in the above proof, if $\chi_1=\chi_2=\chi_3$ are given by $\chi$ as in (\ref{eq:chiex}), then the conditions (1) and (2) from Proposition \ref{prop:ellfin3case2} are satisfied. In other words, $\langle \chi_1,\chi_2,\chi_3\rangle$ contains zero.
\end{example}

\end{document}